\documentclass{amsart}

\usepackage{bbm}
\usepackage{relsize}
\usepackage{amssymb}
\usepackage{enumerate}
\usepackage{mathrsfs}
\usepackage{hyperref}
\usepackage{url}
\usepackage[justification=centering, skip=20pt]{caption}
\usepackage{tikz}
\usepackage[all,cmtip]{xy}
\usepackage{graphicx}

\usepackage{setspace}

\hypersetup{
    colorlinks,
    citecolor=black,
    filecolor=black,
    linkcolor=black,
    urlcolor=black
    }

\usepackage{geometry}    
\newtheorem{thm}{Theorem}
\newtheorem{cor}[thm]{Corollary}
\newtheorem{prop}[thm]{Proposition}
\newtheorem{lem}[thm]{Lemma}
\newtheorem{conj}[thm]{Conjecture}

\newtheorem{defn}[thm]{Definition}

\newtheorem{rem}[thm]{Remark}

\newcommand{\R}[0]{\mathbb{R}}
\newcommand{\Z}[0]{\mathbb{Z}}

\newcommand{\T}[0]{\mathbb{T}}
\newcommand{\C}[0]{\mathbb{C}}
\renewcommand{\t}[1]{\textup{#1}}

\numberwithin{equation}{section}

\newcommand\numberthis{\addtocounter{equation}{1}\tag{\theequation}}

\usepackage{etoolbox}

\makeatletter
\patchcmd{\@settitle}{\uppercasenonmath\@title}{}{}{}
\patchcmd{\@setauthors}{\MakeUppercase}{}{}{}
\patchcmd{\section}{\scshape}{}{}{}
\makeatother

\title[Schr\"odinger equations on compact globally symmetric spaces]{Schr\"odinger equations on compact globally symmetric spaces}

\author[Y. Zhang]{Yunfeng Zhang}
\address{Department of Mathematics, University of Connecticut, Storrs, CT 06269}
\email{yunfeng.zhang@uconn.edu}

\begin{document}

\onehalfspacing

\begin{abstract} 
In this article, we establish scale-invariant Strichartz estimates for the Schr\"odinger equation on arbitrary compact globally symmetric spaces and some bilinear Strichartz estimates on products of rank-one spaces. As applications, we provide local well-posedness results for nonlinear Schr\"odinger equations on such spaces in both subcritical and critical regularities.
\end{abstract}

\maketitle


\section{Introduction}
\subsection{NLS on compact manifolds}
Let $M$ be a Riemannian manifold. 
The Cauchy problem for the nonlinear Schr\"odinger equation (NLS)
\begin{align}\label{Cauchyproblem}
\left\{
\begin{array}{ll}
i\partial_t u+\Delta u=F(u), & \ u=u(t,x),\ t\in\R, \ x\in M,\\
u(0,x)=u_0(x)\in H^s(M), & \ x\in M.
\end{array}
\right.
\end{align}
has been intensively studied in the literature. For simplicity as well as physical applications, the nonlinearity $F$ is usually taken to be a polynomial in $u$ and its complex conjugate $\bar{u}$, or a power like expression as
$$F(u)=\pm |u|^{\beta-1}u.$$
For such nonlinearities, the degree $\beta$ of $F$ may be defined, which contributes to the scaling symmetry of solutions $u(t,x)\mapsto \lambda^{\frac{2}{\beta-1}}u(\lambda^2 t,\lambda x)$ if $M$ is the Euclidean space $\R^d$. The critical Sobolev exponent 
$$s_c=\frac{d}{2}-\frac{2}{\beta-1}$$
leaves the homogeneous Sobolev $\dot{H}^{s_c}(\R^d)$-norm invariant, thus by an analogy to the Euclidean space, for the above Cauchy problem on any manifold $M$, we say the initial datum $u_0$ is of the subcritical regularity if $s>s_c$, critical if $s=s_c$, and supercritical if $s<s_c$. 
In this article, we focus on when $M$ is compact and the local well-posedness theory for the above NLS. By local well-posedness, we usually mean existence and uniqueness of a solution in (a subspace of) $C([-T,T], H^s(M))$ for some $T>0$, continuous dependence of the solution  on the initial data, and persistence of higher-order Sobolev regularities. A general strategy of establishing local well-posedness is to view the solution as the fixed point of the Duhamel operator
$$\Phi(u)=e^{it\Delta}u_0-i\int_0^te^{i(t-s)\Delta}F(u)\ ds,$$
and to establish that $\Phi$ is a contraction mapping on some ball in some suitable function subspace of $C([-T,T], H^s(M))$. This is then usually reduced to establishing linear and multi-linear Strichartz estimates, i.e., space-time estimates in Lebesgue spaces for the Schr\"odinger propagator $e^{it\Delta}$. 

We now review the literature. 
The tori were the first examples of compact manifold studied for the NLS. In the ground breaking paper of Bourgain \cite{Bou93}, there was established scale-invariant Strichartz estimates on square tori $M=\T^d$ of the form 
\begin{align}\label{Stri}
\|e^{it\Delta}f\|_{L^p(I\times M)}\leq C \|f\|_{H^{d/2-(d+2)/p}(M)}
\end{align}
for a range of $p$, where $I$ is a finite interval of $t$. This was then used to establish local well-posedness for the NLS in all subcritical regularities for a range of degree of the nonlinearity, combined with a novel choice of a function space, the so-called $X^{s,b}$-space or Fourier restriction space which measures how far is the solution from a linear evolution. For the critical regularity, the paper \cite{HTT11} of Herr, Tataru and Tzvetkov established local well-posedness as well as global well-posedness for small initial data on the three dimensional square torus for the energy critical exponent $s_c=1$, which serves as the first such critical result on compact manifolds. The main novelties in this paper are a tri-linear scale-invariant Strichartz estimate 
\begin{align}\label{trilinearStri}
\left\|\prod_{i=1}^3e^{it\Delta}f_i\right\|_{L^2(I\times M)}\leq C\left(\frac{N_3}{N_1}+\frac{1}{N_2}\right)^{\delta}N_2N_3\prod_{i=1}^3\|f_i\|_{L^2(M)}
\end{align}
for some $\delta>0$,  
where $f_i$ is spectrally localized in the band $[N_i,2N_i]$ and here it is assumed that $N_1\geq N_2\geq N_3$; and an application of some new function spaces the so-called $U^p$-, $V^p$-spaces which adapt perfectly to the above trilinear estimate. 
Similar results have also been established on irrational rectangular tori, see for example the work \cite{GOW14} of Guo, Oh, and Wang, and \cite{Str14} of Strunk. 
Later, Bourgain and Demeter \cite{BD15} established the full range 
of Strichartz estimates on irrational rectangular tori, with the scaling $\varepsilon$-loss eliminated in \cite{KV16} by Killip and Visan, and consequences of these Strichartz estimates to local well-posedness of NLS can be seen in \cite{KV16} and the work \cite{Lee19} of Lee. 

Much less is known on general compact manifolds. Burq, G\'erard and Tzvetkov established in \cite{BGT04} some non-scale-invariant Strichartz estimates on general compact manifolds as follows 
$$\|e^{it\Delta}f\|_{L^p(I,L^q(M))}\leq C\|f\|_{H^{1/p}(M)}$$
for pairs $(p,q)$ satisfying the following admissibility conditions
$$\frac{d}{2}=\frac{2}{p}+\frac{d}{q}, \ p,q\geq 2, \ (p,q,d)\neq(2,\infty,2).$$
Such admissibility conditions are informed by a scaling consideration, and if the Sobolev space $H^{1/p}(M)$ were replaced by $L^2(M)$, the above estimates become scale-invariant and hold true on the Euclidean spaces. These inequalities were used to derive local well-posedness for a range of subcritical regularity, but because of their scaling loss, they do not cover the full range of subcritical regularity as well as the critical regularity. 
Indeed, on the two sphere $\mathbb{S}^2$ and for the cubic NLS, the critical regularity is $s_c=0$, but as established by Burq, G\'erard and Tzvetkov in \cite{BGT02}, the Cauchy problem is not (uniformly) locally well-posed for any $s<1/4$, while it is locally well-posed for $s>1/4$ as established in \cite{BGT04}. For the proof of this latter result, a key ingredient is a sharp bilinear Strichartz estimate on the two sphere, which was in turn established 
by a sharp bilinear Laplace-Beltrami eigenfunction bound on the two sphere as follows 
$$\|f_1f_2\|_{L^2(M)}\leq C \min(N_1,N_2)^{1/4}\|f_1\|_{L^2(M)}\|f_2\|_{L^2(M)}$$
where $f_i$ is a Laplace-Beltrami eigenfunction with eigenvalue $N_i^2\geq 1$. 
On the three sphere $\mathbb{S}^3$ and the product $\mathbb{S}^2\times\mathbb{S}$ of spheres, following a similar approach,  Burq, G\'erard and Tzvetkov established in \cite{BGT05m} local and global well-posedness for the energy regularity $s=1$ for all $s_c<1$. Then Herr \cite{Her13} established on $\mathbb{S}^3$ and more generally three dimensional Zoll manifolds the critical counterpart, the local well-posedness as well as global well-posedness with small initial data for the energy critical regularity $s=s_c=1$, by a trilinear Strichartz estimate as in \eqref{trilinearStri}, and a key ingredient in getting such an estimate was a tri-linear Laplace-Beltrami eigenfunction bound similar to the above bilinear one as derived in \cite{BGT05m}. Later, the same results were established on $\mathbb{S}^2\times\mathbb{S}$ by Herr and Strunk \cite{HS15}. 

\subsection{Main results}
Apart from tori, spheres, or a product of spheres, local well-posedness results that beat the general results in \cite{BGT04} on compact manifolds seem missing in the literature. In this article, we provide such results on a wide class of compact manifolds, the compact globally symmetric spaces. By a compact globally symmetric space, we mean a compact manifold whose geodesic symmetries are global isometries. Such spaces are well classified: they are either compact Lie groups, or some compact homogeneous spaces; see Helgason's classic \cite{Hel01}. They distinguish themselves from general compact homogeneous spaces by a clearer structure and a harmonic analysis closer to that on Euclidean spaces; see Helgason's other classics \cite{Hel00} and \cite{Hel08}. In particular, any compact globally symmetric space $M$ has an embedded maximal torus, the dimension of which we call the rank of $M$, and this provides a bridge from harmonic analysis on $M$ to that on tori. Using such harmonic analysis, we first establish the following scale-invariant Strichartz estimate on these spaces. 
\begin{thm}\label{Main}
The scale-invariant estimate \eqref{Stri} holds for any $p\geq 2+\frac{8}{r}$ on any compact globally symmetric space of dimension $d$ and rank $r$, equipped with the canonical Killing metric.
\end{thm}
A similar result was established on the special class of compact Lie groups in \cite{Zha20}. We will prove this theorem following a framework informed by Bourgain's work \cite{Bou93} on tori, while the key Lie theoretic input is an integral formula of spherical functions on compact globally symmetric spaces established by Clerc in \cite{Cle88}, similar to that of Harish-Chandra on noncompact symmetric spaces. As a byproduct of the proof of this theorem, we also provide the following $L^p$-bounds of Laplace-Beltrami eigenfunctions on compact globally symmetric spaces of high rank, which improve upon Sogge's classical general bound \cite{Sog88} on compact manifolds and match those established by Bourgain \cite{Bou93e} on tori. 

\begin{thm}\label{eigenfunctionbound}
On any compact globally symmetric space $M$ of dimension $d$ and rank $r\geq 5$, suppose $f$ is an eigenfunction of the Laplace-Beltrami operator of eigenvalue $-N^2\leq -1$. Then 
\begin{align}\label{eigen}
\|f\|_{L^p(M)}\leq C N^{\frac{d-2}{2}-\frac{d}{p}}\|f\|_{L^2(M)}
\end{align}
for any $p>2+\frac{8}{r-4}$. 
\end{thm}

We conjecture however that the Strichartz estimate \eqref{Stri} should hold for any $p>\frac{2(d+2)}{d}$ on compact globally symmetric spaces of rank at least two, and the above eigenfunction bound \eqref{eigen} should hold for any $p>\frac{2d}{d-2}$. 
Because of the scale-invariant nature of the Strichartz estimates established in Theorem \ref{Main}, we are able to provide the following local well-posedness result for the NLS that covers all subcritical regularities, in a range of degree of the nonlinearity. 

\begin{thm}
Suppose in \eqref{Cauchyproblem}, $F(u)$ is a polynomial function in $u$ and its complex conjugate $\bar{u}$ of degree $\beta$ such that $F(0)=0$. Then on any compact globally symmetric space of dimension $d$ and rank $r$ and for any 
$\beta\geq 3+8/r$, the Cauchy problem
\eqref{Cauchyproblem} is uniformly locally well-posedness for all $s>s_c=\frac{d}{2}-\frac{2}{\beta-1}$.
\end{thm} 

Note that the above theorem does not cover the cubic nonlinearity $\beta=3$, for which we provide another approach via bilinear Strichartz estimates and give local well-posedness results for both subcritical and critical regularities. We will establish the following bilinear Strichartz estimates on certain compact globally symmetric spaces. 

\begin{thm}
Suppose $f_i\in L^2(M)$ is spectrally localized in the band $[N_i,2N_i]$, $i=1,2$. \\
(i) Suppose either $M=M_1\times\cdots\times M_r$ is a product of rank-one compact globally symmetric spaces such that $r\geq 2$, or $M=\T^{r_0}\times M_1\times\cdots\times M_{r-r_0}$ is a product of a rational $r_0$-dimensional torus $\T^{r_0}$ and rank-one spaces such that $r\geq 3$, and in both cases we assume that each $M_i$ has the dimension at least 3. Then  
$$\|e^{it\Delta}f_1\ e^{it\Delta}f_2\|_{L^2(I\times M)}\leq C_\varepsilon\min(N_1,N_2)^{\frac{d}{2}-1+\varepsilon} \|f_1\|_{L^2(M)}\|f_2\|_{L^2(M)}.$$
(ii) Suppose either $M=M_1\times\cdots\times M_r$ is a product of rank-one spaces, or $M=\T^{r_0}\times M_1\times\cdots\times M_{r-r_0}$ is a product of a rational torus $\T^{r_0}$ and rank-one spaces, and in both cases we assume $r\geq 3$ and that each component $M_i$ has the dimension at least 4. Suppose $N_1\geq N_2$. Then for some $\delta>0$ 
$$\|e^{it\Delta}f_1\ e^{it\Delta}f_2\|_{L^2(I\times M)}\leq C\left(\frac{N_2}{N_1}+\frac{1}{N_2}\right)^{\delta} N_2^{\frac{d}{2}-1} \|f_1\|_{L^2(M)}\|f_2\|_{L^2(M)}.$$
\end{thm}

The approaches of multi-linear Strichartz estimates as in \cite{BGT05, BGT05m, Her13, HS15} all explored the multi-linear bounds for the Laplace-Beltrami operator, which work well for spheres. However, this is only typical for the rank-one case, but on higher-rank spaces, the full ring of invariant differential operators should play the role of the Laplace-Beltrami operator on rank-one spaces. So instead we establish the above results by establishing  bilinear bounds for joint eigenfunctions of the ring of invariant differential operators on $M$, and the above limitation on $M$ is due to our inability to establish such bilinear joint eigenfunction bounds in full generality. As consequences of the above theorem, we derive local well-posedness results for the cubic NLS as follows. 

\begin{thm}
Suppose in \eqref{Cauchyproblem}, $F(u)$ equals any of $\pm|u|^2u$, $\pm u^3$, $\pm |u|^2\bar{u}$, $\pm \bar{u}^3$. Then on such $M$ as assumed in (i) of the above theorem, the Cauchy problem \eqref{Cauchyproblem} is locally well-posedness for all $s>s_c=\frac{d}{2}-1$, and on such $M$ as assumed in (ii) of the above theorem, the Cauchy problem is locally well-posedness for $s=s_c$. 
\end{thm}

The organization of the rest of this paper is as follows. We review harmonic analysis on compact globally symmetric spaces in Section \ref{Clerc}. In particular we review Clerc's formula of spherical functions in Section \ref{structure} and \ref{spherical}, and obtain a Fourier support lemma for a product of joint eigenfunctions in Section \ref{support}. Then we prove the linear Strichartz estimates in Section \ref{linearStri}. In Section \ref{bilinear}, we derive bilinear Strichartz estimates. In Section \ref{fun}, we apply these Strichartz estimates to local well-posedness results. 

For notations, we use $A\lesssim B$ to mean $A\leq CB$ for some positive constant $C$, $A\lesssim_x B$ to mean $A\leq C(x)B$ for some positive constant $C(x)$ depending on $x$, and $A\asymp B$ to mean $|A|\lesssim |B|\lesssim |A|$.

\section{Analysis on compact globally symmetric spaces}\label{Clerc}

In this section, we review the structure of and harmonic analysis on compact globally symmetric spaces.

\subsection{Structure theory}\label{structure}
Let $\mathfrak{g}$ be a real semisimple Lie algebra with the Cartan decomposition 
$$\mathfrak{g}=\mathfrak{k}+\mathfrak{p}$$
associated to a Cartan involution $\theta$ of $\mathfrak{g}$. 
Let $\mathfrak{a}\subset\mathfrak{p}$ be a maximal abelian subspace. Let $\mathfrak{g}_\C$ be the complexification of $\mathfrak{g}$. If $\mathfrak{z}$ is a subalgebra of $\mathfrak{g}$, let $\mathfrak{z}_\C$ denote its complexification in $\mathfrak{g}_\C$. 
Let 
$$\mathfrak{u}=\mathfrak{k}+ i\mathfrak{p}.$$
Let $\mathfrak{a}^*$ denote the real dual of $\mathfrak{a}$, and let $\Sigma\subset\mathfrak{a}^*$ denote the restricted root system with respect to $(\mathfrak{g},\mathfrak{a})$ and $\Sigma^+$ the set of positive roots. For $\alpha\in\Sigma$, let $\mathfrak{g}_\alpha$ denote the associated root subspace, and let 
$$\mathfrak{n}=\sum_{\alpha\in\Sigma^+}\mathfrak{g}_\alpha.$$
Let $\mathfrak{m}\subset\mathfrak{k}$ be the centralizer of $\mathfrak{p}$ in $\mathfrak{k}$, and let $\mathfrak{d}$ be a maximal abelian subspace of $\mathfrak{m}$. Then 
$$\mathfrak{h}=\mathfrak{d}+i\mathfrak{a}$$ is a Cartan subalgebra of $\mathfrak{u}$. Let $G_\C$ be the simply connected complex Lie group with Lie algebra $\mathfrak{g}_\C$ and let $G,K,A,N,K_\mathbb{C}, A_\mathbb{C},N_\mathbb{C},U$ be the analytic subgroups corresponding to the subalgebras $\mathfrak{g},\mathfrak{k},\mathfrak{a},\mathfrak{n},\mathfrak{k}_\C,\mathfrak{a}_\C,\mathfrak{n}_\C,\mathfrak{u}$ respectively. 

\begin{defn}\label{defncptsym}
$U/K$ is called a symmetric space of compact type. A compact Riemannian manifold $M$ is said to be a compact globally symmetric space, or compact symmetric space in short, if $M$ is a symmetric space of compact type, or more generally, if $M$ is a product of a symmetric space of compact type and a torus. We equip symmetric spaces of compact type the canonical metric induced from the Killing form, and we assume the tori are of the form $\R^{r_0}/2\pi \Z^{r_0}$ equipped with a rational metric such that $(\xi,\eta)\in\mathbb{Q}$ for any $\xi,\eta\in\Z^{r_0}$. 
\end{defn}

Let 
$$F=\{a\in A_\C:\ a^2=e\}.$$


We have the Iwasawa decomposition 
$$K\times A\times N\xrightarrow{\sim} G, \ (k,a,n)\mapsto kan.$$
Consider its complexification
$$K_\C\times A_\C\times N_\C\to G_\C, \ (k,a,n)\mapsto kan.$$
This map is however neither injective nor surjective. The following two theorems clarify this matter. Let $\omega\subset G_\C$ be the image of this multiplication map. For $g\in G$, let $k(g)\in K$, $a(g)\in A$, $n(g)\in N$ denote the elements in the Iwasawa decomposition $g=k(g)a(g)n(g)$. 

\begin{thm}\cite[Theorem 1.10]{Cle88} \t{(See also }\cite{SW02}\t{)}
$\omega$ is a dense open subset of $G_\C$. 
\end{thm}

\begin{thm}\cite[Theorem 1.11]{Cle88}
The map $g\mapsto n(g)$ is extended to a holomorphic function on $\omega$, with values in $N_\C$. The maps $g\mapsto a(g)$ and $g\mapsto k(g)$ are extended to multivalued holomorphic functions on $\omega$, with values in $A_\C$ and $K_\C$ respectively; for $g_0\in\omega$ and for any branch $a_1$ of $a$ defined in a neighborhood of $g_0$, there is a unique branch $k_1$ of $k$ such that $g=k_1(g)a_1(g)n(g)$. If $a_2$ is another branch and $k_2$ is the corresponding branch, then there exists $d\in F$ such that $k_1=k_2d$ and $a_1=a_2d=da_2$. 
\end{thm}

By the above theorem, for $g\in \omega$, we let $\mathcal{H}(g)$ denote the multivalued holomorphic function with values in $\mathfrak{a}_\C$, defined by 
$$\exp\mathcal{H}(g)=a(g).$$

\subsection{Clerc's formula of spherical functions}\label{spherical}

Let $(\cdot,\cdot)$ denote the Killing form, then we have the weight lattice
$$
\Lambda=\left\{\lambda\in\mathfrak{a}^*:\ \frac{(\lambda,\alpha)}{(\alpha,\alpha)}\in\Z, \forall \alpha\in\Sigma\right\}
$$
and the subset of dominant weights 
$$
\Lambda^+=\left\{\lambda\in\mathfrak{a}^*:\ \frac{(\lambda,\alpha)}{(\alpha,\alpha)}\in\Z_{\geq 0}, \forall \alpha\in\Sigma^+\right\}.$$
It is well-known \cite[Theorem 4.1 of Chapter V]{Hel00} that $\Lambda^+$ indexes the equivalence classes of irreducible spherical representations as follows. An irreducible representation of $U$ is called spherical if there exists a nonzero vector fixed by $K$.  
Each irreducible spherical representation has a highest weight $\lambda\in i\mathfrak{h}^*$ which vanishes on $\mathfrak{d}$, hence we may view $\lambda$ as lying in $\mathfrak{a}^*$, thus giving an element in $\Lambda^+$.

For $\lambda\in\Lambda^+$, let $\pi_\lambda$ be the corresponding irreducible spherical representation in the inner product space $(V_\lambda,\langle\ ,\ \rangle)$, and let $e_\lambda\in V_\lambda$ be a $K$-invariant unit vector. 
Let
$$\varphi_\lambda(g)=\langle\pi_\lambda(g) e_\lambda,e_\lambda\rangle$$ 
be the corresponding spherical function on $U$. $\varphi_\lambda$ is in fact bi-invariant under $K$, thus it can be considered as a $K$-invariant function on the compact symmetric space $U/K$. 
$\varphi_\lambda$ is extended to a holomorphic function on $G_\C$. Let $dk$ denote the normalized Haar measure on $K$. 

\begin{thm}\cite[Theorem 2.2, Corollary 2.4]{Cle88}\label{Cle88} \t{(Clerc's formula)}
For any $g\in G_\C$,
\begin{align*}
K_g=\{k\in K:\ gk\in\omega\}
\end{align*}
is open in $K$ and $K\setminus K_g$ is of zero measure in $K$. We have   
\begin{align}\label{Clerc's}
\varphi_\lambda(g)=\int_{K_g}e^{\lambda(\mathcal{H}(gk))}\ dk, \ \forall g\in G_\C.
\end{align} 
Moreover, for $u\in U$, 
\begin{align}\label{phase}
\t{Re }\lambda(\mathcal{H}(uk))\leq 0, \ \forall k\in K. 
\end{align}
\end{thm}

The above theory may seem a little too abstract, so we elucidate it by the example of 
$$\mathfrak{g}=\mathfrak{sl}_2(\mathbb{R})=\left\{\begin{pmatrix} a & b \\ c & -a \end{pmatrix}:\ a,b,c\in\mathbb{R}\right\},$$
equipped with the Cartan involution $\theta: X\to -X^T$. Then we have $G=\t{SL}_2(\mathbb{R})$ and $G_\C=\t{SL}_2(\mathbb{C})$.
We calculate
$$\mathfrak{k}=\mathfrak{so}_2(\mathbb{R})=\left\{\begin{pmatrix} 0 & a \\ -a & 0 \end{pmatrix}:\ a\in\mathbb{R}\right\},\ K=\t{SO}_2(\mathbb{R}),\ K_\C=\t{SO}_2(\mathbb{C}),$$
$$\mathfrak{u}=\mathfrak{su}_2=\left\{\begin{pmatrix} ia & b+ic \\ -b+ic & -ia  \end{pmatrix}:\ a,b,c\in\mathbb{R}\right\},\ U=\t{SU}_2,$$
$$\mathfrak{a}=\left\{\begin{pmatrix} a & 0 \\ 0 & -a \end{pmatrix}:\ a\in\mathbb{R}\right\},\ A=\left\{\begin{pmatrix} e^t & 0 \\ 0 & e^{-t} \end{pmatrix}:\ t\in\mathbb{R}\right\},\ A_\mathbb{C}=\left\{\begin{pmatrix} e^z & 0 \\ 0 & e^{-z} \end{pmatrix}:\ z\in\mathbb{C}\right\},$$
$$\mathfrak{n}=\left\{\begin{pmatrix} 0 & a \\ 0 & 0 \end{pmatrix}:\ a\in\mathbb{R}\right\},\ N=\left\{\begin{pmatrix} 0 & x \\ 0 & 0 \end{pmatrix}:\ x\in\mathbb{R}\right\},
\ N_\C=\left\{\begin{pmatrix} 0 & z \\ 0 & 0 \end{pmatrix}:\ z\in\mathbb{C}\right\},
$$
$$F=\left\{\begin{pmatrix} 1 & 0 \\ 0 & 1 \end{pmatrix},\ \begin{pmatrix} -1 & 0 \\ 0 & -1 \end{pmatrix}\right\},$$
$$\omega=\left\{\begin{pmatrix} a & b \\ c & d \end{pmatrix}\in \t{SL}_2(\mathbb{C}):\ a^2+c^2\neq 0\right\}.$$
The associated compact symmetric space $U/K=\t{SU}_2/\t{SO}_2(\mathbb{R})$ is simply the two sphere. 
For $g=\begin{pmatrix} a & b \\ c & d \end{pmatrix}\in \omega$, we have 
$$\mathcal{H}(g)=\log \sqrt{\begin{pmatrix} a^2+c^2 & 0 \\ 0 & (a^2+c^2)^{-1} \end{pmatrix}}.$$
Let $\alpha:\mathfrak{a}_\C\to \C$ be defined by 
$$\alpha\begin{pmatrix} a & 0 \\ 0 & -a \end{pmatrix}=2a, \ \forall a\in\C.$$
Then 
$$\Lambda^+=\{\lambda_n:=n\alpha:\ n\in\mathbb{Z}_{\geq 0}\},$$
so that 
$$e^{\lambda_n(\mathcal{H}(g))}=(a^2+c^2)^n.$$
A maximal torus of $U=\t{SU}_2$ is 
$$B=\left\{\begin{pmatrix} e^{i\theta} & 0 \\ 0 & e^{-i\theta} \end{pmatrix}:\ \theta\in \mathbb{R}\right\}.$$ 
Let $\varphi_n$ denote the spherical function associated to the highest weight $\lambda_n$ for $n\in\mathbb{Z}_{\geq 0}$. Since $\varphi_n$ is bi-invariant under $K$, to determine its values it suffices to restrict it on $B$. For $u=\begin{pmatrix} e^{i\theta} & 0 \\ 0 & e^{-i\theta} \end{pmatrix}\in B$ and $k=\begin{pmatrix} \cos t & \sin t \\ -\sin t & \cos t \end{pmatrix}\in K$, 
$$uk=\begin{pmatrix} e^{i\theta}\cos t & e^{i\theta}\sin t \\  -e^{-i\theta}\sin t  & e^{-i\theta}\cos t \end{pmatrix},$$ 
so that 
$$e^{\lambda_n(\mathcal{H}(uk))}=(e^{2i\theta}\cos^2t+e^{-2i\theta}\sin^2t)^n=(\cos 2\theta+i\sin 2\theta\cos 2t)^n.$$
Identifying $K\cong \mathbb{R}/2\pi \mathbb{Z}$, we have $dk=\frac{1}{2\pi }dt$. By Clerc's formula \eqref{Clerc's}, we have 
$$\varphi_n(\theta)=\frac{1}{2\pi}\int_0^{2\pi}(\cos 2\theta+i\sin 2\theta\cos 2t)^n\ dt.$$
Thus we have retrieved Laplace's formula for the Legendre polynomials.

Using Clerc's formula, we can rewrite sums of the form $\sum_{\lambda\in\Lambda^+}f_\lambda\varphi_\lambda(g)$ into
\begin{align*}
\sum_{\lambda\in\Lambda^+}f_\lambda\varphi_\lambda(g)=\int_{K_g}\left(\sum_{\lambda\in\Lambda^+}f_\lambda e^{\lambda(\mathcal{H}(gk))}\right)\ dk.
\end{align*}
Thus to derive pointwise estimates of the above sum on the left side, it suffices to get estimates on the inner sum on the right side of the above equation. We will follow this method to derive estimates for the kernel function associated to the Schr\"odinger propagator $e^{it\Delta}$, in order to prove Theorem \ref{Main}.

\subsection{Fourier series and Fourier support of products}\label{support}
In this subsection, we review some basics about Fourier series on a compact symmetric space. The analogue of Fourier series on tori for an arbitrary compact manifold may be the spectral decomposition of the Laplace-Beltrami operator. Though it is very useful for a generic manifold, it does not by itself capture the symmetries of manifold, for example those behind the high rank of a symmetric space. For a compact symmetric space $U/K$, a finer analogue of Fourier series may be the spectral decomposition of the full ring $D(U/K)$ of $U$-invariant differential operators as follows. For continuous functions $f$ and $g$ on $U$, let $f*g$ denote their convolution with respect to a normalized Haar measure on $U$. If $f,g$ are both right-invariant under $K$, then so is $f*g$. 
For $\lambda\in\Lambda^+$, let $d_\lambda$ denote the dimension of the irreducible spherical representation $\pi_\lambda$ of $U$. For $f\in L^2(U/K)$, 
define the spectral projectors
$$P_\lambda f=d_\lambda f*\varphi_\lambda.$$
Then we have the Fourier series 
$$f=\sum_{\lambda\in\Lambda^+} P_\lambda f.$$
This is also the spectral decomposition of $D(U/K)$, where $P_\lambda f$ is a joint eigenfunction of $D(U/K)$ of spectral parameter $\lambda$, and is also equal to a matrix coefficient of the form 
$\langle \pi_\lambda v_\lambda,e_\lambda\rangle$ for some $v_\lambda\in V_\lambda$. 
By Schur's orthogonality relations, for $f_\lambda, f_\mu$ joint eigenfunctions of distinct spectral parameter $\lambda,\mu$ respectively, we have 
$$\langle f_\lambda, f_\mu \rangle_{L^2(U/K)}=0.$$
In particular, we have Parseval's identity
$$\langle f, g\rangle_{L^2(U/K)}=\sum_{\lambda\in\Lambda^+}\langle P_\lambda f, P_\lambda g\rangle_{L^2(U/K)}.$$
Now suppose $\kappa$ is a $K$-invariant function on $U/K$. Then $\kappa$ has a Fourier transform $\hat{\kappa}$ defined as follows 
\begin{align}\label{FT}
\kappa=\sum_{\lambda\in \Lambda^+}\hat{\kappa}(\lambda)d_\lambda\varphi_\lambda.
\end{align}
Then we have 
$$\|f*\kappa\|_{L^2(U/K)}\leq\sup_{\lambda}|\hat{\kappa}(\lambda)|\cdot \|f\|_{L^2(U/K)}.$$


On a torus, the Fourier support of the product of two functions is the sum of the Fourier support of each; we also have the following analogous fact for symmetric spaces of compact type. We say $f\in L^2(U/K)$ is Fourier supported on a subset $S$ of $\Lambda^+$, if $P_\lambda f$ vanishes for $\lambda\notin S$. Let $|\cdot|$ denote the norm on $\mathfrak{a}^*$ induced from the Killing form. Let $w_1,\ldots,w_r\in\Lambda^+$ be the fundamental weights, and let 
$$\rho_0=\sum_{i=1}^r w_i.$$


\begin{lem}[Fourier support of products]\label{Fouriersupport}
For $f,g\in L^2(U/K)$ and $\lambda,\mu\in \Lambda^+$, we have that $P_\lambda f \cdot P_\mu g$ is Fourier supported on 
$$\{\nu\in\Lambda^+:\ \nu=\lambda+\xi,\ \xi\in\Lambda, \ |\xi+\rho_0|\leq |\mu+\rho_0|\}.$$
\end{lem}
\begin{proof}
$P_\lambda f\cdot P_\mu g$ appears as a matrix coefficient of the tensor product representation $\pi_\lambda\otimes\pi_\mu$. 
By Steinberg's formula 
and its consequences \cite[Exercise 12 of Chapter VI Section 24]{Hum72}, we have 
$$\pi_\lambda\otimes\pi_\mu=\bigoplus_{\nu\in\Lambda^+, \ \nu=\lambda+\xi, \ \xi\in\Pi(\mu)}m_\nu\pi_\nu,$$
where $\Pi(\mu)\subset \Lambda$ is the set of weights associated to $\pi_\mu$, and $m_\nu$ is the multiplicity of $\pi_\nu$ in $\pi_\lambda\otimes\pi_\mu$ which is allowed to be zero. This implies 
$$P_\lambda f\cdot P_\mu g=\sum_{\nu\in\Lambda^+, \ \nu=\lambda+\xi, \ \xi\in\Pi(\mu)}f_\nu,$$
where $f_\nu$ is a joint eigenfunction of spectral parameter $\nu$. By Lemma C of Chapter III Section 13.4 of \cite{Hum72}, we have $|\xi+\rho_0|\leq |\mu+\rho_0|$ for all $\xi\in\Pi(\mu)$, and this finishes the proof. 
\end{proof}

More generally, we write out the Fourier series on a compact symmetric space $M$ that has a toric component. Suppose $M=\T^{r_0}\times M'$ where $\T^{r_0}$ is a rational torus of rank $r_0$ and $M'=U/K$ is a symmetric space of compact type. 
For $f\in L^2(\T^{r_0}\times M')$, we first write the Fourier series with respect to the variable $x\in \T^{r_0}$
$$f(x,y)=\sum_{\xi\in\Z^{r_0}}e^{i(\xi,x)}\hat{f}(\xi,y),$$
and then write the Fourier series with respect to the variable $y\in M'$ described as above 
\begin{align}\label{Fourierseries}
f(x,y)=\sum_{\xi\in\Z^{r_0},\lambda\in\Lambda^+}e^{i(\xi,x)}P_{\lambda}\hat{f}(\xi,y),
\end{align}
which is the Fourier series of $f$ on $M$. For a $K$-invariant function $\kappa$ on $M$, it has a Fourier transform $\hat{\kappa}$ defined as follows 
\begin{align}\label{Fouriertransform}
\kappa=\sum_{\xi\in\Z^{r_0},\ \lambda\in \Lambda^+}\hat{\kappa}(\xi,\lambda)e^{i(\xi,x)}d_\lambda\varphi_\lambda.
\end{align}
Then we have 
$$\|f*\kappa\|_{L^2(M)}\leq\sup_{\xi,\lambda}|\hat{\kappa}(\xi,\lambda)|\cdot \|f\|_{L^2(M)}.$$

\subsection{Littlewood-Paley projections} \label{reduction}
Now we introduce Littlewood-Paley projections. Let 
$$\rho=\frac{1}{2}\sum_{\alpha\in\Sigma^+}\dim(\mathfrak{g}_\alpha) \alpha,$$
and let $\Delta$ denote the Laplace-Beltrami operator. Then for any joint eigenfunction $f_\lambda$ of spectral parameter $\lambda$, we have  
$$\Delta f_\lambda=(-|\lambda+\rho|^2+|\rho|^2)f_\lambda.$$
Let 
$$|\lambda|_\rho=\sqrt{|\lambda+\rho|^2-|\rho|^2}, \ \forall\lambda\in\Lambda^+.$$
Let $N=2^m$, $m=0,1,\ldots$. For $f\in L^2(U/K)$, define the sharp Littlewood-Paley projections 
$$P_N f=\sum_{N\leq |\lambda|_\rho<2N} P_\lambda f.$$ 
We may also define the mollified version as follows. Pick $\phi_0\in C_c^\infty(\mathbb{R})$ and $\phi\in C_c^\infty(\mathbb{R}\setminus\{0\})$ to satisfy $\phi_0(y)+\sum_{m=0}^\infty \phi(2^{-m}y)=1$ for all $y\in \mathbb{R}$, then we define 
$$\tilde{P}_N f=\sum_{\lambda\in\Lambda^+} \phi\left(|\lambda|_\rho/N\right)P_\lambda f.$$ 
More generally, for $M=\T^{r_0}\times M'$, using \eqref{Fourierseries}, we have 
$$\Delta f=\sum_{\xi\in\Z^{r_0},\ \lambda\in\Lambda^+}
(-|\xi|^2-|\lambda|_\rho^2)e^{i(\xi,x)}P_\lambda\hat{f}(\xi,y).$$
Then define 
$$P_Nf=\sum_{N\leq \sqrt{|\xi|^2+|\lambda|_\rho^2}<2N}e^{i(\xi,x)}P_\lambda\hat{f}(\xi,y),$$
$$\tilde{P}_Nf=\sum_{\xi\in\Z^{r_0},\ \lambda\in\Lambda^+}\phi\left( \sqrt{|\xi|^2+|\lambda|_\rho^2}\mathlarger{\mathlarger{/}}N\right)e^{i(\xi,x)}P_\lambda\hat{f}(\xi,y).$$



\section{Linear Strichartz estimates}\label{linearStri}
\subsection{The Schr\"odinger kernel}
By Littlewood-Paley theory on compact manifolds \cite[Corollary 2.3]{BGT04}, Strichartz estimates \eqref{Stri} may be reduced to 
\begin{align}\label{StriMollified}
\|\tilde{P}_Ne^{it\Delta}f\|_{L^p(I\times M)}\lesssim N^{\frac{d}{2}-\frac{d+2}{p}}\|f\|_{L^2(M)}
\end{align}
or equivalently 
\begin{align*}
\|P_Ne^{it\Delta}f\|_{L^p(I\times M)}\lesssim N^{\frac{d}{2}-\frac{d+2}{p}}\|f\|_{L^2(M)}.
\end{align*}
We will use the mollified version for the proof of Theorem \ref{Main}. We now assume that $M=U/K$ is a symmetric space of compact type, and the cases when $M$ has a toric component can be established by a straightforward adaptation. 
Let $f\in L^2(U/K)$. Expressing the Schr\"odinger propagator as a convolution
\begin{align*}
\tilde{P}_Ne^{it\Delta}f=f*\mathscr{K}_N(t,\cdot),
\end{align*} 
the kernel function $\mathscr{K}_N(t,\cdot)$ reads 
\begin{align}\label{mollifiedkernel}
\mathscr{K}_N(t,x)=\sum_{\lambda\in\Lambda^+}\phi(|\lambda|_\rho/N) e^{-it|\lambda|^2_\rho}d_\lambda\varphi_\lambda. 
\end{align}
We need the following well-known fact in the theory of root systems. 
\begin{lem}[Rationality of the Killing form] \label{rationality}
$(\lambda,\mu)\in \mathbb{Q}$ whenever $\lambda$ and $\mu$ lie in the weight lattice or the root lattice. 
\end{lem}
By the above lemma, 
there exists $\mathcal{T}\in 2\pi\Z$, such that 
\begin{align}\label{D}
|\lambda|_\rho^2\in \frac{2\pi}{\mathcal{T}}\mathbb{Z},\ \forall\mu\in\Lambda^+,
\end{align} 
thus both $\mathscr{K}$ and $\mathscr{K}_N$ are actually periodic in $t$ with the period $\mathcal{T}$. Thus we may think of the variable $t$ as lying on the torus
$\mathbb{T}=\R/\mathcal{T}\Z$. This key observation allows us to replace the time interval $I$ by $\mathbb{T}$, and to use the $TT^*$ argument to transform \eqref{StriMollified} into 
\begin{align}\label{spacetimeestimate}
\|\mathscr{K}_N * F\|_{L^p(\mathbb{T}\times M)}\lesssim N^{d-\frac{2(d+2)}{p}}\|F\|_{L^{p'}(\mathbb{T}\times M)}. 
\end{align}
Here $1/p+1/p'=1$ and $*$ stands for the convolution on the product symmetric space $\mathbb{T}\times M$. The above estimate can be thought of as a (discrete) Fourier restriction estimate for the parabola 
$$\left\{(n,\lambda)\in \frac{2\pi\mathbb{Z}}{\mathcal{T}}\times \Lambda^+:\ n=|\lambda|_\rho^2\right\}$$ 
on such product spaces. Following the framework of Bourgain \cite{Bou93}, we now wish to derive pointwise estimate of the Schr\"odinger kernel $\mathscr{K}_N$. 
Using Clerc's formula \eqref{Clerc's}, we rewrite the Schr\"odinger kernel 
\begin{align}\label{Clerc'sapp}
\mathscr{K}_N(t,uK)=\int_{K_u}\kappa_{N}(t,u,k)\ dk,
\end{align}
where
\begin{align}\label{defnofkappaN}
\kappa_{N}(t,u,k)=\sum_{\lambda\in\Lambda^+}\phi\left(|\lambda|_\rho/N\right)e^{-it(|\mu+\rho|^2-|\rho|^2)+\mu(\mathcal{H}(uk))}d_\mu. 
\end{align}
Then
\begin{align}\label{Kk}
|\mathscr{K}_N(t,uK)|\leq\sup_{k\in K}|\kappa_{N}(t,u,k)|.
\end{align}
Note that $\kappa_N(t,u,k)$ is in the form of a Weyl type exponential sum, so we wish to apply the classical technique of Weyl differencing to bound it on major arcs of the time variable as follows. 

\begin{prop}\label{dispersive}
Let $\|\cdot\|$ stand for the distance from the nearest integer. Define the major arcs
\begin{align*}
\mathcal{M}_{a,q}=\left\{s\in\mathbb{R}/\Z:\ \left\|s-\frac{a}{q}\right\|<\frac{1}{qN}\right\}
\end{align*}
where
\begin{align*}
a\in\mathbb{Z}_{\geq 0},\ q\in\mathbb{Z}_{\geq 1},\ a<q,\ (a,q)=1,\ q<N. 
\end{align*}
Then 
\begin{align}\label{kappaN}
|\kappa_N(t,u,k)|\lesssim\frac{N^{d}}{\left(\sqrt{q}\left(1+N\left\|\frac{t}{\mathcal{T}}-\frac{a}{q}\right\|^{1/2}\right)\right)^r}
\end{align}
for $\frac{t}{\mathcal{T}}\in\mathcal{M}_{a,q}$, uniformly for $u\in U$ and $k\in K$. 
\end{prop}
In particular, by \eqref{Kk}, the above proposition implies  
\begin{align}\label{KNtH}
|\mathscr{K}_N(t,uK)|\lesssim\frac{N^{d}}{\left(\sqrt{q}\left(1+N\left\|\frac{t}{\mathcal{T}}-\frac{a}{q}\right\|^{1/2}\right)\right)^r}
\end{align}
for $\frac{t}{\mathcal{T}}\in\mathcal{M}_{a,q}$, uniformly for $uK\in U/K$.

\subsection{Proof of Proposition \ref{dispersive}}\label{kernelbound}
First, we record some technical facts concerning $d_\lambda$ ($\lambda\in\Lambda^+$), which are crucial in avoiding any scaling loss in the following proof of \eqref{kappaN}. Let $\Phi\subset \mathfrak{h}_\C^*$ denote the root system associated to $(\mathfrak{g}_\C,\mathfrak{h}_\C)$  and let $\Phi^+\subset\Phi$ be the set of positive roots with respect to an ordering compatible with that on $\mathfrak{a}^*$ (this ordering determines the positive roots in the restricted root system $\Sigma$). As explained in the previous section, the highest weight $\lambda\in\Lambda^+$ is originally an element in $i\mathfrak{h}^*$, and we may also extend it linearly as an element in the complex dual $\mathfrak{h}_\C^*$ of $\mathfrak{h}_\C$, and we still denote this element as $\lambda$. 

\begin{lem}
\begin{align*}
d_\lambda=\frac{\prod_{\alpha\in\Phi^+, \alpha|_\mathfrak{a}\neq 0}(\lambda+\rho', \alpha)}{\prod_{\alpha\in\Phi^+, \alpha|_\mathfrak{a}\neq 0}(\rho', \alpha)}, \ \text{ where }\rho'=\frac{1}{2}\sum_{\alpha\in\Phi^+}\alpha.
\end{align*}
\end{lem}
\begin{proof}
The irreducible spherical representation associated to $\lambda\in\Lambda^+$ induces the irreducible representation of $\mathfrak{g}_\C$ of the highest weight $\lambda\in\mathfrak{h}^*_\C$. Apply Weyl's dimension formula to this representation, we have 
$$d_\lambda=\frac{\prod_{\alpha\in\Phi^+}(\lambda+\rho', \alpha)}{\prod_{\alpha\in\Phi^+}(\rho', \alpha)}.$$ 
It suffices to observe that if $\alpha|_\mathfrak{a}=0$, then $(\lambda,\alpha)=0$, since $\lambda|_\mathfrak{d}=0$ where $\mathfrak{d}$ is the orthogonal complement of $i\mathfrak{a}$ in $\mathfrak{h}$. 
\end{proof}

This implies the following.

\begin{cor}\label{dmu}
Let $\gamma_1,\ldots,\gamma_r$ be the simple roots in $\Sigma^+$. 
Let $w_1,\ldots,w_r\in\mathfrak{a}^*$ be the fundamental weights such that 
$$\frac{(w_i,\gamma_j)}{(\gamma_j,\gamma_j)}=\delta_{ij}$$ for $1\leq i,j\leq r$. 
Then for $\lambda=n_1w_1+\cdots+n_rw_r\in\Lambda^+$, $d_\lambda$ is a polynomial in $n_1,\ldots,n_r$ of degree $d-r$, $d$ being the dimension of the symmetric space $U/K$. Furthermore, $d_\lambda$ has at least one linear factor of the form $a_jn_j+b_j$ for each $j=1,\ldots,r$, with $a_j,b_j$ constant. 
\end{cor}

\begin{proof}
By the above Lemma, the degree of $d_\lambda$ in $\lambda$ equals the number of restricted roots counted with multiplicities, which is equal to $d-r$, and we have  
\begin{align}
d_\lambda=\frac{\prod_{\alpha\in\Phi^+,\alpha|_\mathfrak{a}\neq \gamma_j,\forall j}(\lambda+\rho',\alpha)}{\prod_{\alpha\in\Phi^+, \alpha|_\mathfrak{a}\neq 0}(\rho', \alpha)}\cdot\prod_{1\leq j\leq r}\prod_{\alpha\in\Phi^+,\alpha|_\mathfrak{a}=\gamma_j}(n_j(\gamma_j,\gamma_j)+(\rho',\alpha)).
\end{align}
Then the second claim is clear. 
\end{proof}

For $k\in K$ and $u\in U$, write 
\begin{align*}
\mathcal{H}(uk)=H_0+iH, \ \t{for }H_0,H\in\mathfrak{a}. 
\end{align*}
By \eqref{phase}, 
\begin{align}\label{phase1}
\lambda(H_0)\leq 0,\ \forall\lambda\in\Lambda^+.
\end{align}
Perform Weyl's differencing technique to \eqref{defnofkappaN}, we have
\begin{align*}
|\kappa_{N}|^2&=\sum_{\lambda_1\in\Lambda^+,\lambda_2\in\Lambda^+}
\phi\left(\frac{|\lambda_1|_\rho}{N}\right)\phi\left(\frac{|\lambda_2|_\rho}{N}\right)e^{-it(|\lambda_1|^2_\rho-|\lambda_2|^2_\rho)+(\lambda_1+\lambda_2)(H_0)+i(\lambda_1-\lambda_2)(H)}d_{\lambda_1}d_{\lambda_2}\\
&=\sum_{\substack{\lambda\in\Lambda^+\\(\lambda=\lambda_1+\lambda_2)}}
e^{it|\lambda|^2+\lambda(H_0)-i\lambda(H)}
\sum_{\substack{\mu\in\Lambda^+\cap (\lambda-\Lambda^+)\\ (\mu=\lambda_1)}}
\phi\left(\frac{|\mu|_\rho}{N}\right)\phi\left(\frac{|\lambda-\mu|_\rho}{N}\right)e^{2i[\mu(H)-t(\mu,\lambda+2\rho)]}d_{\mu}d_{\lambda-\mu}\\
&\lesssim \sum_{\lambda\in\Lambda^+,|\lambda|\lesssim N}
\left|
\sum_{\mu\in\Lambda^+\cap (\lambda-\Lambda^+)}
\phi\left(\frac{|\mu|_\rho}{N}\right)\phi\left(\frac{|\lambda-\mu|_\rho}{N}\right)e^{2i[\mu(H)-t(\mu,\lambda+2\rho)]}d_{\mu}d_{\lambda-\mu}
\right|. \numberthis \label{sum}
\end{align*}

Here, we have crucially used \eqref{phase1}. Recall that $w_1,w_2,\ldots,w_r$ denote the fundamental dominant weights so that 
$$\Lambda^+=\Z_{\geq 0}w_1+\cdots+\Z_{\geq 0}w_r$$ 
and 
$$\Lambda=\Z w_1+\cdots+\Z w_r.$$
For $\lambda\in\Lambda^+$, write 
\begin{align*}
\lambda=n_1^\lambda w_1+\cdots+n_r^\lambda w_r. 
\end{align*}
Then  
\begin{align*}
\mu\in\Lambda^+\cap(\lambda-\Lambda^+) \t{ if and only if }\mu=n_1w_1+\cdots+n_rw_r, \ 0\leq n_j\leq n_j^\lambda, \ j=1,\ldots, r. 
\end{align*}
For $\mu=n_1w_1+\cdots+n_rw_r$, let 
\begin{align*}
f(n_1,\ldots,n_r)=f(\mu)=f(\mu,N,\lambda):=\phi\left(\frac{|\mu|_\rho}{N}\right)\phi\left(\frac{|\lambda-\mu|_\rho}{N}\right)d_{\mu}d_{\lambda-\mu}.
\end{align*}
Define the difference operators with respect to $\mu$
\begin{align*}
D_jf(\mu):=f(\mu+w_j)-f(\mu), \ j=1,\ldots,r. 
\end{align*}

\begin{lem} \label{difference}

Let $n,n_1,\ldots,n_r\in\Z_{\geq 0}$ and $j_1,\ldots,j_n\in\{1,\ldots,r\}$. We have the following facts. 

\noindent (i) $
|D_{j_1}D_{j_2}\cdots D_{j_n}f(n_1,\ldots,n_r)|\lesssim_n N^{2d-2r-n}.
$

\noindent (ii) Fix $J\subset\{1,\ldots,r\}$. Suppose $j_k\notin J$, for any $k=1,\ldots,n$. Suppose for each $j\in J$, either $n_j=0$ or $n_j=n_j^\lambda+1$. Then 
\begin{align*}
|D_{j_1}D_{j_2}\cdots D_{j_n}f(n_1,\ldots,n_r)|\lesssim_n N^{2d-2r-n-|J|}.
\end{align*}
\end{lem}
\begin{proof}
Since $\phi$ is a bump function, we may assume $|\mu|,|\lambda|\lesssim N$. The Leibniz rule for taking difference operators is
\begin{align*}
D_j\left(\prod_{i=1}^nf_i\right)=
\underbrace{\sum_{k=1}^l (D_jf_k)\prod_{i\neq k}f_i}_{\t{``main term''}}+
\underbrace{\sum_{l=2}^n\sum_{1\leq k_1<\cdots<k_l\leq n}
(D_jf_{k_1})\cdots (D_jf_{k_l})\cdot \prod_{\substack{i \neq k_1,\cdots, k_l\\ 1\leq i\leq n}}f_i}_{\t{``higher order terms''}}.
\end{align*}
Then the inequality in (i) is a consequence of the fact from Corollary \ref{dmu} that $d_\mu$ and $d_{\lambda-\mu}$ are polynomial functions in $\mu$ of degree $d-r$. For (ii), compared with (i), the extra decay factor of $N^{-|J|}$ results from the fact from Corollary \ref{dmu} that $d_\mu$ (resp. $d_{\lambda-\mu}$) has a linear factor of the form $a_jn_j+b_j$ (resp. $a_j(n_j^\lambda-n_j)+b_j$) for each $j=1,\ldots,r$. Indeed, for each $j\in J$, as $n_j=0$ (resp. $n_j=n_j^\lambda+1$), this linear factor is bounded $\lesssim 1$. It is assumed that no $D_{j}$ ($j\in J$) appears in the inequality in (ii), thus for each $j\in J$, this linear factor survives the difference operators and gives a contribution of $\lesssim 1$ compared to that of $\lesssim N$ in (i), so it provides an extra decay factor of $N^{-1}$ compared to (i), and together yielding an extra decay factor of $N^{-|J|}$.   
\end{proof}

Now let $\kappa^{\lambda}$ be the sum in \eqref{sum} inside of the absolute value. Then 
\begin{align*}
\kappa^{\lambda}=\sum_{0\leq n_r\leq n^\lambda_r}e^{in_r\theta_r}\cdots\sum_{0\leq n_1\leq n^\lambda_1} e^{in_1\theta_1} f(\mu)
\end{align*}
where 
$
\theta_j=\theta_{j}(t,H,\lambda):=2[w_j(H)-t(w_j,\lambda+2\rho)], \ j=1,\ldots,r
$. 
We perform summation by parts on $\kappa^\lambda$ choicefully: 
\begin{itemize} 
\item If $|1-e^{i\theta_1}|\geq  N^{-1}$, 
we write
\begin{align*}
\sum_{0\leq n_1\leq n^\lambda_1} e^{in_1\theta_1} f(\mu)
&=\frac{1}{1-e^{i\theta_1}}\sum_{0\leq n_1\leq n_1^\lambda}e^{i(n_1+1)\theta_1}D_1f(\mu)\\
&\ \ +\frac{1}{1-e^{i\theta_1}}f(0,n_2,\ldots,n_r)-\frac{e^{i(n_1^\lambda+1)\theta_1}}{1-e^{i\theta_1}}f(n_1^\lambda+1,n_2,\ldots,n_r)\\
&=\underbrace{\frac{1}{(1-e^{i\theta_1})^2}\sum_{0\leq n_1\leq n_1^\lambda}e^{i(n_1+2)\theta_1}D_1^2f(\mu)}_{l_1\t{ such terms appear in }\kappa_{\t{term}}^\lambda}\\
&\ \ +\underbrace{\frac{e^{i\theta_1}}{(1-e^{i\theta_1})^2}D_1f(0,n_2,\ldots,n_r)}_{l_2\t{ such terms appear in }\kappa_{\t{term}}^\lambda}-\underbrace{\frac{e^{i(n_1^\lambda+2)\theta_1}}{(1-e^{i\theta_1})^2}D_1f(n_1^\lambda+1,n_2,\ldots,n_r)}_{l_3\t{ such terms appear in }\kappa_{\t{term}}^\lambda}\\
&\ \ +\underbrace{\frac{1}{1-e^{i\theta_1}}f(0,n_2,\ldots,n_r)}_{l_4\t{ such terms appear in }\kappa_{\t{term}}^\lambda}-\underbrace{\frac{e^{i(n_1^\lambda+1)\theta_1}}{1-e^{i\theta_1}}f(n_1^\lambda+1,n_2,\ldots,n_r)}_{l_5\t{ such terms appear in }\kappa_{\t{term}}^\lambda}, \numberthis \label{errorterms}
\end{align*}
\item Otherwise, no operation is performed on $\underbrace{\sum_{0\leq n_1\leq n^\lambda_1} e^{in_1\theta_1} f(\mu)}_{l_6\t{ such terms appear in }\kappa_{\t{term}}^\lambda}$. 
\end{itemize}
Next similarly we perform summation by parts choicefully for each of the layers of summation 
$$\sum_{0\leq n_j\leq n_j^{\lambda}}e^{in_j\theta_j}\cdots, \ j=2,\ldots,r.$$ 
What we end up with is a sum of $\leq 5^r$ terms, each term of the form 
\begin{align*}
\kappa^\lambda_{\t{term}}=\frac{\sum_{\substack{0\leq n_{j_k}\leq n_{j_k}^\lambda\\ 1\leq k\leq l_1\t{ or }l_1+l_2+l_3+l_4+l_5+1\leq k\leq r}} z\cdot D_{j_1}^2\cdots D_{j_{l_1}}^2 D_{j_{l_1+1}}\cdots D_{j_{l_1+l_2+l_3}}f(\mu_*)}{\prod_{1\leq k\leq l_1+l_2+l_3}(1-e^{i\theta_{j_k}})^2\prod_{l_1+l_2+l_3+1\leq k\leq l_1+l_2+l_3+l_4+l_5}(1-e^{i\theta_{j_k}})},
\end{align*}
where 
$$l_1+l_2+l_3+l_4+l_5+l_6=r$$ 
is a partition of $r$ into six nonnegative integers, and 
\begin{align}\label{mustar}
\mu_*=\sum_{1\leq k\leq l_1} n_{j_k}w_{j_k} +\sum_{\substack{l_1+l_2+1\leq k\leq l_1+l_2+l_3,\\ \t{or }l_1+l_2+l_3+l_4+1\leq k\leq l_1+l_2+l_3+l_4+l_5}} (n_{j_k}^{\lambda}+1)w_{j_k},
\end{align}
with 
\begin{align*}
|z|=1.
\end{align*}

Using \eqref{mustar}, apply (ii) of Lemma \ref{difference} for $J=\{j_k,\ l_1+l_2+l_3+1\leq k\leq l_1+l_2+l_3+l_4+l_5\}$, we get 

\begin{align*}
|\kappa^\lambda_{\t{term}}|&\lesssim \frac{
N^{l_1+l_6+2d-2r-2l_1-l_2-l_3-l_4-l_5}}{\prod_{1\leq k\leq l_1+l_2+l_3}|1-e^{i\theta_{j_k}}|^2\prod_{l_1+l_2+l_3+1\leq k\leq l_1+l_2+l_3+l_4+l_5}|1-e^{i\theta_{j_k}}|}\\
&\lesssim \frac{
N^{l_1+l_6+2d-2r-2l_1-l_2-l_3-l_4-l_5}}{\prod_{1\leq k\leq l_1+l_2+l_3}|1-e^{i\theta_{j_k}}|^2\prod_{l_1+l_2+l_3+1\leq k\leq l_1+l_2+l_3+l_4+l_5}|1-e^{i\theta_{j_k}}|^2}\\
&\lesssim N^{2d-3r}\prod_{j=1}^r\frac{1}{\max\{\frac{1}{N},|1-e^{i\theta_j}|\}^2}\\
&\lesssim N^{2d-3r}\prod_{j=1}^r\frac{1}{\max\{\frac{1}{N},\|\theta_j/2\pi\|\}^2}.
\end{align*}
The same estimates hold for $|\kappa^\lambda|$. 
Now a good choice of $\mathcal{T}$ in \eqref{D} also makes 
$-2(w_j,\lambda+2\rho)\in \frac{2\pi}{\mathcal{T}}\Z$, $j=1,\ldots,r$.  
Let 
\begin{align*}
m_j=-2(w_j,\lambda+2\rho)\cdot \frac{\mathcal{T}}{2\pi}, \ j=1,\ldots,r.
\end{align*} 
Since the map $\Lambda\ni\lambda\mapsto (m_1,\ldots,m_r)\in \Z^r$ is one-one, we can write \eqref{sum} into 
\begin{align*}
|\kappa_{N}|^2\lesssim N^{2d-3r}\prod_{j=1}^r\left(\sum_{|m_j|\lesssim N}\frac{1}{\max\{\frac{1}{N},\|m_jt/\mathcal{T}+w_j(H)/\pi\|\}^2}\right). 
\end{align*}
By a standard estimate as in deriving the classical Weyl sum estimate in one dimension, we get 
\begin{align*}
\sum_{|m_j|\lesssim N}\frac{1}{\max\{\frac{1}{N},\|m_jt/\mathcal{T}+w_j(H)/\pi\|\}^2}\lesssim\frac{N^3}{[\sqrt{q}(1+N\|\frac{t}{\mathcal{T}}-\frac{a}{q}\|^{1/2})]^2}
\end{align*}
for $\frac{t}{\mathcal{T}}$ lying on the major arc $\mathcal{M}_{a,q}$. Hence
\begin{align*}
|\kappa_{N}(t,u,k)|^2\lesssim \frac{N^{2d}}{[\sqrt{q}(1+N\|\frac{t}{\mathcal{T}}-\frac{a}{q}\|^{1/2})]^{2r}}.
\end{align*}
An inspection of the above argument shows that this estimate holds uniformly for $u\in U$ and $k\in K$. This finishes the proof of Proposition \ref{dispersive} and thus \eqref{KNtH}.

\begin{rem}
The last piece of information in Corollary \ref{dmu} about linear factors of $d_\mu$ is used to deal with the last two terms in \eqref{errorterms}. Without it, one could only get the bound 
$$|\kappa^\lambda|\lesssim N^{2d-2r}\prod_{j=1}^r\frac{1}{\max\{\frac{1}{N},\|\theta_j/2\pi\|\}},$$
which would only yield \eqref{kappaN} with an $N^\varepsilon$ loss. 
\end{rem}

\subsection{Farey dissection}\label{farey} 
With the kernel bound \eqref{KNtH} at hand, we may establish 
Theorem \ref{Main} as well as Theorem \ref{eigenfunctionbound}
using the circle method of Farey dissection, as we did in \cite{Zha201} on the special case of compact Lie groups. 
For expository purposes, in the next few sections we include these proofs. The circle method of Farey dissection is a slightly simpler version of the circle method than that originally applied in \cite{Bou93}, in that it does not involve any minor arcs. We review it in this section. 
Let $n$ be an integer and consider the Farey sequence 
$$\left\{\frac{a}{q},\ a\in\mathbb{Z}_{\geq 0},\ q\in\mathbb{Z}_{\geq 1}, \ (a,q)=1, \ a<q, \ q\leq n\right\}$$ of order $n$ on the unit circle. For each three consecutive fractions $\frac{a_l}{q_l}$, $\frac{a}{q}$, $\frac{a_r}{q_r}$ in the sequence, consider the Farey arc 
$$\mathcal{M}_{a,q}=\left[\frac{a_l+a}{q_l+q},\frac{a+a_r}{q+q_r}\right]$$
around $\frac{a}{q}$. The Farey dissection $\bigsqcup_{a,q} \mathcal{M}_{a,q}$ of order $n$ of the unit circle has the important uniformity property that both $\left[\frac{a_l+a}{q_l+q},\frac{a}{q}\right]$ and $\left[\frac{a}{q},\frac{a+a_r}{q+q_r}\right]$ are of length $\asymp \frac{1}{qn}$ \cite[Theorem 35]{HW08}. We make a further dissection of the unit circle as follows. Fix a large number $N$ and let $Q$ be dyadic integers, i.e. powers of 2, between 1 and $N$. Consider the Farey sequence of order $\lfloor N\rfloor$. For $Q\leq q<2Q$, let $L$ denote dyadic integers between $Q$ and $N$, and we decompose the Farey arc into a disjoint union 
$$\mathcal{M}_{a,q}=\bigsqcup_{Q\leq L\leq N}\mathcal{M}_{a,q,L}$$
where $\mathcal{M}_{a,q,L}$ is an interval on the unit circle of the form $\left\|t-\frac{a}{q}\right\|\asymp \frac{1}{NL}$, except when $L$ is the largest dyadic integer $\leq N$, $\mathcal{M}_{a,q,L}$ is defined as an interval of the form $\left\|t-\frac{a}{q}\right\|\lesssim \frac{1}{N^2}$. Let $\mathbbm{1}_{Q,L}$ denote the indicator function of the subset 
$$\mathcal{M}_{Q,L}=\bigsqcup_{0\leq a<q, \ (a,q)=1, \ Q\leq q<2Q}\mathcal{M}_{a,q,L}$$
of the unit circle, then we have a partition of unity 
$$1=\sum_{Q,L}\mathbbm{1}_{Q,L}\left(\frac{t}{\mathcal{T}}\right)$$
on the circle $\T=\mathbb{R}/\mathcal{T}\Z$. 
Let $\widehat{\mathbbm{1}_{Q,L}}$ denote the Fourier transform of $\mathbbm{1}_{Q,L}\left(\frac{t}{\mathcal{T}}\right)$ on $\T$ such that 
$$\mathbbm{1}_{Q,L}\left(\frac{t}{\mathcal{T}}\right)=\sum_{n\in\frac{2\pi }{\mathcal{T}}\Z}
\widehat{\mathbbm{1}_{Q,L}}(n)e^{int},$$
then
\begin{align}\label{1QM}
\|\widehat{\mathbbm{1}_{Q,L}}\|_{l^\infty}\lesssim \|\mathbbm{1}_{Q,L} \|_{L^1(\mathbb{T})}\lesssim\frac{Q^2}{NL}. 
\end{align}

\subsection{Proof of Theorem \ref{Main}}
We write 
$$\mathscr{K}_N(t,x)=\sum_{Q,L}\mathscr{K}_{Q,L}(t,x), \ \mathscr{K}_{Q,L}(t,x):=\mathscr{K}_N(t,x)\cdot\mathbbm{1}_{Q,L}\left(\frac{t}{\mathcal{T}}\right),$$
for $(t,x)\in\T\times M$. Let $F: \T \times M\to \C$ be a continuous function. Let $*$ denote the convolution on the product symmetric space $\T\times M$. 
Using \eqref{KNtH}, we have 
\begin{align*}
\|F*\mathscr{K}_{Q,L}\|_{L^{\infty}(\T\times M)}&\leq \|\mathscr{K}_{Q,L}\|_{L^{\infty}(\T\times M)}\|F\|_{L^{1}(\T\times M)}\\
&\lesssim N^{d-\frac{r}{2}}L^{\frac{r}{2}}Q^{-\frac{r}{2}}\|F\|_{L^{1}(\T\times M)}. \numberthis \label{2p'to2p}
\end{align*}
On the other hand, 
as a $K$-invariant function, $\mathscr{K}_{Q,L}$ has its Fourier transform $\widehat{\mathscr{K}_{Q,L}}$ on $\mathbb{T}\times M$ computed as follows (see \eqref{Fouriertransform})
$$\widehat{\mathscr{K}_{Q,L}}(n,\mu)=\phi(|\mu|_\rho/N)\widehat{\mathbbm{1}_{Q,L}}\left(n+|\mu|_\rho^2\right), \ \t{for }(n,\mu)\in \frac{2\pi\Z}{\mathcal{T}}\times \Lambda^+.$$
By \eqref{1QM}, we have 
$$\sup_{n,\mu}\left|\widehat{\mathscr{K}_{Q,L}}(n,\mu)\right|\lesssim \frac{Q^2}{NL}.$$
As a consequence, we have 
\begin{align}\label{2to2}
\|F*\mathscr{K}_{Q,L}\|_{L^{2}(\T\times U)}
\lesssim \frac{Q^2}{NL}\|F\|_{L^2(\T\times U)}. 
\end{align}
Interpolating \eqref{2p'to2p} with \eqref{2to2} for 
$\frac{\theta}{2}=\frac{1}{p}$, 
we get 
\begin{align*}
\|F*\mathscr{K}_{Q,L}\|_{L^p(\T\times U)}\lesssim N^{\left(d-\frac{r}{2}\right)(1-\theta)-\theta}L^{\frac{r}{2}(1-\theta)-\theta}Q^{-\frac{r}{2}(1-\theta)+2\theta}\|F\|_{L^{p'}(\T\times U)}.
\end{align*}
We require the exponent of $Q$ satisfy 
\begin{align*}
-\frac{r}{2}(1-\theta)+2\theta<0 \Leftrightarrow  \theta<\frac{r}{4+r},
\end{align*}
which implies the exponent of $M$ satisfies 
$
\frac{r}{2}(1-\theta)-\theta>0
$.
Summing over the dyadic integers $M$ and $Q$, we get 
\begin{align*}
\|F*\mathscr{K}_N\|_{L^p(\T\times U)}&\lesssim \sum_{1\leq Q\leq N}\sum_{Q\leq L\leq N}\|F*\mathscr{K}_{Q,L}\|_{L^p(\T\times U)}\lesssim N^{d(1-\theta)-2\theta}\|F\|_{L^{p'}(\T\times U)}=N^{d-\frac{2(d+2)}{p}}\|F\|_{L^{p'}(\T\times U)},
\end{align*}
provided 
$$\frac{1}{p}=\frac{\theta}{2}<\frac{r}{2(4+r)} \Leftrightarrow p>2+\frac{8}{r}.$$

\subsection{Proof of Theorem \ref{eigenfunctionbound}}
We also include the proof of Theorem \ref{eigenfunctionbound} for the convenience of the reader even though it is not applied elsewhere in this paper. Let $f$ be an eigenfunction of eigenvalue $-N^2$. Then $N=|\mu|_\rho$ for some $\mu\in\Lambda^+$.   
Set 
$$\mathcal{K}_N=\sum_{\mu\in\Lambda^+,|\mu|_\rho=N}d_\mu \chi_\mu.$$ 
Then it is clear that $f=f*\mathcal{K}_N$. By an argument of $TT^*$, it suffices to establish bounds of the form 
$$\|f*\mathcal{K}_N\|_{L^p(M)}\lesssim N^{\frac{d-2}{2}-\frac{d}{p}}\|f\|_{L^{p'}(M)}.$$
Let $\mathscr{K}_{N}$ be again the mollified Schr\"odinger kernel \eqref{mollifiedkernel}, and here we assume that the cutoff function $\phi$ satisfies $\phi(1)=1$. Then we may write 
$$\mathcal{K}_N=\frac{1}{\mathcal{T}}\int_0^\mathcal{T}\mathscr{K}_N(t,\cdot)e^{it N^2}\ dt.$$
Using Farey dissection again, we decompose 
$$\mathcal{K}_N=\sum_{Q,L}\mathcal{K}_{Q,L},$$
where 
$$\mathcal{K}_{Q,L}=\int_{\mathcal{M}_{Q,L}}\mathscr{K}_{Q,L}(t,\cdot) e^{it N^2}\ d\left(\frac{t}{\mathcal{T}}\right).$$
By the kernel bound \eqref{KNtH}, Minkowski's integral inequality, and the estimate that the length $|\mathcal{M}_{Q,L}|$ of $\mathcal{M}_{Q,L}$ is bounded by $\lesssim\frac{Q^2}{NL}$, we have 
\begin{align}\label{KQMestimate}
\|\mathcal{K}_{Q,L}\|_{L^\infty(M)}\lesssim N^{d-\frac{r}{2}-1}L^{\frac{r}{2}-1}Q^{-\frac{r}{2}+2},
\end{align}
which implies
\begin{align}\label{2u'to2u}
\|f*\mathcal{K}_{Q,L}\|_{L^{\infty}(M)}\lesssim N^{d-\frac{r}{2}-1}L^{\frac{r}{2}-1}Q^{-\frac{r}{2}+2}\|f\|_{L^{1}(M)}. 
\end{align}
On the other hand, the Fourier transform of $\mathcal{K}_{Q,L}$ on $M$ equals (see \eqref{FT})
$$\widehat{\mathcal{K}_{Q,L}}(\mu)=\phi(|\mu|_\rho/N)\int_{\mathcal{M}_{Q,L}}e^{it(N^2-|\mu|_\rho^2)}\ d\left(\frac{t}{\mathcal{T}}\right), \ \t{for all }\mu\in\Lambda^+.$$
Thus 
$$\sup_\mu \left|\widehat{\mathcal{K}_{Q,L}}(\mu)\right|\lesssim |\mathcal{M}_{Q,L}|\lesssim \frac{Q^2}{NL},$$
which implies 
\begin{align}\label{2to22}
\|f*\mathcal{K}_{Q,L}\|_{L^2(U)}
\lesssim 
\frac{Q^2}{NL}\|f\|_{L^2(U)}.
\end{align}
Interpolating \eqref{2u'to2u} with \eqref{2to22} for 
$\frac{\theta}{2}=\frac{1}{p}$, 
we get 
\begin{align*}
\|f*\mathcal{K}_{Q,L}\|_{L^p(U)}\lesssim N^{\left(d-\frac{r}{2}-1\right)(1-\theta)-\theta} L^{\left(\frac{r}{2}-1\right)(1-\theta)-\theta}Q^{\left(-\frac{r}{2}+2\right)(1-\theta)+2\theta}\|f\|_{L^{p'}(U)}. 
\end{align*}
We require the exponent of $Q$ be negative, i.e., 
$$\left(-\frac{r}{2}+2\right)(1-\theta)+2\theta<0\Leftrightarrow \theta<\frac{r-4}{r},$$
which implies that the exponent $\left(\frac{r}{2}-1\right)(1-\theta)-\theta$ of $M$ is positive. Summing over the dyadic integers $M$ and $Q$, we have
\begin{align*}
\|f*\mathcal{K}_\lambda\|_{L^p(U)}\lesssim N^{\left(d-2\right)(1-\theta)-2\theta}\|f\|_{L^{p'}(U)}=N^{d-2-\frac{2d}{p}}\|f\|_{L^{p'}(U)},
\end{align*} 
provided 
$$\frac{1}{p}=\frac{\theta}{2}
<\frac{r-4}{2r}\Leftrightarrow p>2+\frac{8}{r-4}.$$

\section{Bilinear Strichartz estimates}\label{bilinear}

In this section, we discuss an approach to bilinear Strichartz estimates, aiming at establishing local well-posedness results for the cubic nonlinearity. Even though the following discussion should in principle apply equally well to tri- and multi-linear Strichartz estimates which would give similar well-posedness results for polynomial nonlinearities of any odd degree, they are more technical to establish which we choose to avoid in this paper. For compact symmetric spaces of higher rank, we argue that the following conjecture on bilinear estimates for the joint eigenfunctions of invariant differential operators should be the key replacement or generalization of those for the Laplace-Beltrami eigenfunctions as discussed in Burq, G\'erard, and Tzvetkov's work \cite{BGT05,BGT05m}. Let $M$ be a symmetric space of compact type. Then after taking a finite cover of $M$ if necessary, it becomes a product of irreducible components.

\begin{conj}\label{bilineareigen}
Suppose $M$ is a symmetric space of compact type of dimension $d$ and rank $r$. \\
(i) For each irreducible component $M_0$ of $M$, suppose that its dimension $d_0$ and rank $r_0$ satisfy $d_0\geq 3r_0$. 
Then for any $f,g\in L^2(M)$ and $\lambda_1,\lambda_2\in\Lambda^+$, we have 
\begin{align}\label{bilinearbilinear}
\|P_{\lambda_1} f\cdot P_{\lambda_2} g\|_{L^2(M)}\lesssim_{\varepsilon} (\min(|\lambda_1|,|\lambda_2|)+1)^{\frac{d}{2}-r+\varepsilon}\|f\|_{L^2(M)}\|g\|_{L^2(M)}. 
\end{align}
(ii) If we assume the stronger condition $d_0> 3r_0$ for each irreducible component $M_0$, then the above estimate holds without the $\varepsilon$-loss. 
\end{conj}

Taking $r=1$, the above statement becomes the bilinear Laplace-Beltrami eigenfunction bounds as established in \cite{BGT05m}. In general, the non-triviality of the above conjecture may already be seen by taking $f=g$ and $\lambda=\mu$ in the above estimate, which yields
\begin{align}\label{L4}
\|P_\lambda f\|_{L^4(M)}\lesssim_{\varepsilon} (|\lambda|+1)^{\frac{d-r}{2}-\frac{d}{4}+\varepsilon}\|f\|_{L^2(M)};
\end{align}
however, even such linear joint eigenfunction bounds are still open. In fact, Marshall in \cite{Mar16} established such linear joint eigenfunction bounds for a full range of Lebesgue exponent as follows
\begin{align}\label{conditionaljointbound}
\| P_\lambda f\|_{L^p(M)}\lesssim (|\lambda|+1)^{s(p)}\|f\|_{L^2(M)},
\end{align}
where 
$$
s(p)=\left\{\begin{array}{ll}
\frac{d-r}{2}-\frac{d}{p}, \ &\text{for }p>\frac{2(d+r)}{d-r},\\
\frac{d-r}{2}\left(\frac{1}{2}-\frac{1}{p}\right), \ &\text{for }2\leq p< \frac{2(d+r)}{d-r};
\end{array}\right.
$$
but they were established only under the regularity condition on the spectral parameter $\lambda$ that $\lambda$ should lie in a fixed cone centered at the origin and away from the walls of the Weyl chamber. To eliminate this regularity constaint seems a nontrivial problem, as indicated in \cite{Mar16}; we also refer to Li's work \cite{Li19} on symmetric spaces of noncompact type that has the potential of eliminating this regularity constraint for eigenfunction bounds on compact locally symmetric spaces. The exponents in the above conditional estimate \eqref{conditionaljointbound} as well as those in the bilinear Laplace-Beltrami eigenfunction bounds in \cite[Theorem 3]{BGT05m} are also the reasons of the conditions $d_0\geq 3r_0$ or $d_0> 3r_0$ assumed in the above conjecture.

We now verify this conjecture when $M$ is a product of rank-one compact symmetric spaces. In this case, the ring of invariant differential operators is simply generated by the Laplace-Beltrami operators on each component of $M$. In fact, we prove the following more general theorem that establishes bilinear estimates for the joint spectral projector. The proof is not difficult as an adaptation of the proof of \cite[Theorem 3]{BGT05m}, but it indicates the difficulties for a possible extension to multilinear estimates, and we choose to include it. 

\begin{prop}\label{jointspectral}
Let $M=M_1\times M_2\times\cdots\times M_r$ be a product of compact Riemann manifolds. Let $\Delta_i$ denote the Laplace-Beltrami operator on $M_i$, let $\chi_i$ be a bump function on $\R$, and let $\lambda_i\geq 1$, $i=1,\ldots,r$. Let $\lambda=(\lambda_1,\ldots,\lambda_r)$, and define the joint spectral projector around $\lambda$ as follows
$$\chi_\lambda=\prod_{i=1}^r\chi_{i}(\sqrt{-\Delta_i}-\lambda_i).$$ 
Set $|\lambda|=\sqrt{\lambda_1^2+\cdots+\lambda_r^2}$. \\
(i) Suppose the dimension of each $M_i$ is at least 3. 
Then for any $\lambda,\mu\in \mathbb{R}_{\geq 1}^r$, 
$$\|\chi_\lambda f\ \chi_\mu g\|_{L^2(M)}\lesssim_\varepsilon \min(|\lambda|,|\mu|)^{\frac{d}{2}-r+\varepsilon}\|f\|_{L^2(M)}\|g\|_{L^2(M)}.$$
(ii) Suppose the dimension of each $M_i$ is at least 4. 
Then for any $\lambda,\mu\in \mathbb{R}_{\geq 1}^r$, 
$$\|\chi_\lambda f\ \chi_\mu g\|_{L^2(M)}\lesssim \min(|\lambda|,|\mu|)^{\frac{d}{2}-r}\|f\|_{L^2(M)}\|g\|_{L^2(M)}.$$
\end{prop}
\begin{proof}
We treat (ii), and the proof of (i) will be entirely similar. We first treat the case $r=2$ and then indicate the changes for the general cases. 
By Lemma 2.3 of \cite{BGT05m}, for $f_i\in L^2(M_i)$ and 
$$\chi_{i}(\sqrt{-\Delta_i}-\lambda_i)f_i=\frac{1}{2\pi}\int_\varepsilon^{2\varepsilon}e^{-i\lambda_i\tau_i}\widehat{\chi_i}(\tau_i)(e^{i\tau_i\sqrt{-\Delta_i}}f_i)\ d\tau_i,$$
we have 
$$\chi_{i}(\sqrt{-\Delta_i}-\lambda_i) f_i=\lambda_i^{\frac{d_i-1}{2}}T_{i,\lambda_i}f_i+R_{i,\lambda_i}f_i,$$
with 
\begin{align}\label{Rif}
\|R_{i,\lambda_i}f_i\|_{H^k(M_i)}\lesssim_{N,k}\lambda_i^{k-N}\|f_i\|_{L^2(M_i)}, \ k=0,\ldots,N,
\end{align}
and the following expression in local coordinates of $T_{i,\lambda_i}f_i$ on each small neighborhood $V_i$ in $M_i$
$$T_{i,\lambda_i}f_i(x_i)=\int_{\mathbb{R}^{d_i}}e^{i\lambda_i\varphi_i(x_i,y_i)}a_i(x_i,y_i,\lambda_i)f_i(y_i)\ dy_i.$$
where $a_i(x_i,y_i,\lambda_i)$ is a polynomial in $\lambda^{-1}_i$  with smooth coefficients supported in the set 
$$\{(x_i,y_i)\in V_i\times V_i:\ |x_i|\leq\delta\lesssim \varepsilon/C\leq|y_i|\leq C\varepsilon \}$$
and $-\varphi_i(x_i,y_i)$ is the geodesic distance on $M_i$ between $x_i$ and $y_i$; $i=1,2$. Thus for 
$$\chi_\lambda=\chi_{1}(\sqrt{-\Delta_1}-\lambda_1)\chi_{2}(\sqrt{-\Delta_2}-\lambda_2)$$ and $f\in L^2(M_1\times M_2)$, we may write
$$\chi_\lambda f=\lambda_1^{\frac{d_1-1}{2}}\lambda_2^{\frac{d_2-1}{2}}T_{1,\lambda_1}T_{2,\lambda_2}f+\lambda_1^{\frac{d_1-1}{2}}T_{1,\lambda_1}R_{2,\lambda_2}f+\lambda_2^{\frac{d_2-1}{2}}R_{1,\lambda_1}T_{2,\lambda_2}f+R_{1,\lambda_1}R_{2,\lambda_2}f,$$
as well as for $\chi_\mu g$, and then 16 products of functions from the right side of the above equation would come out in the product $\chi_\lambda f\ \chi_\mu g$, and it would suffice to prove the desired bound for each of these products. However, due to similarity of these products and commutativity between $T_{1,\lambda_1},R_{1,\lambda_1}$ and $T_{2,\lambda_2},R_{2,\lambda_2}$, it suffices to prove the following 7 bounds
\begin{align}\label{TTTT}
\|T_{1,\lambda_1}T_{2,\lambda_2}f\ T_{1,\mu_1}T_{2,\mu_2}g\|_{L^2(M)}\lesssim \min(|\lambda|,|\mu|)^{\frac{d}{2}-r}(\lambda_1\mu_1)^{-\frac{d_1-1}{2}}(\lambda_2\mu_2)^{-\frac{d_2-1}{2}}\|f\|_{L^2(M)}\|g\|_{L^2(M)},
\end{align}
\begin{align}\label{TTTR}
\|T_{1,\lambda_1}T_{2,\lambda_2}f\ T_{1,\mu_1}R_{2,\mu_2}g\|_{L^2(M)}\lesssim  \min(|\lambda|,|\mu|)^{\frac{d}{2}-r}(\lambda_1\mu_1)^{-\frac{d_1-1}{2}}\lambda_2^{-\frac{d_2-1}{2}}\|f\|_{L^2(M)}\|g\|_{L^2(M)},
\end{align}
\begin{align}\label{TRTR}
\|T_{1,\lambda_1}R_{2,\lambda_2}f\ T_{1,\mu_1}R_{2,\mu_2}g\|_{L^2(M)}\lesssim \min(|\lambda|,|\mu|)^{\frac{d}{2}-r}(\lambda_1\mu_1)^{-\frac{d_1-1}{2}}\|f\|_{L^2(M)}\|g\|_{L^2(M)},
\end{align}
\begin{align}\label{TTRR}
\|T_{1,\lambda_1}T_{2,\lambda_2}f\ R_{1,\mu_1}R_{2,\mu_2}g\|_{L^2(M)}\lesssim \min(|\lambda|,|\mu|)^{\frac{d}{2}-r}\lambda_1^{-\frac{d_1-1}{2}}\lambda_2^{-\frac{d_2-1}{2}}\|f\|_{L^2(M)}\|g\|_{L^2(M)},
\end{align}
\begin{align}\label{TRRT}
\|T_{1,\lambda_1}R_{2,\lambda_2}f\ R_{1,\mu_1}T_{2,\mu_2}g\|_{L^2(M)}\lesssim \min(|\lambda|,|\mu|)^{\frac{d}{2}-r}\lambda_1^{-\frac{d_1-1}{2}}\mu_2^{-\frac{d_2-1}{2}}\|f\|_{L^2(M)}\|g\|_{L^2(M)},
\end{align}
\begin{align}\label{TRRR}
\|T_{1,\lambda_1}R_{2,\lambda_2}f\ R_{1,\mu_1}R_{2,\mu_2}g\|_{L^2(M)}\lesssim \min(|\lambda|,|\mu|)^{\frac{d}{2}-r}\lambda_1^{-\frac{d_1-1}{2}}\|f\|_{L^2(M)}\|g\|_{L^2(M)},
\end{align}
\begin{align}\label{RRRR}
\|R_{1,\lambda_1}R_{2,\lambda_2}f\ R_{1,\mu_1}R_{2,\mu_2}g\|_{L^2(M)}\lesssim  \min(|\lambda|,|\mu|)^{\frac{d}{2}-r}\|f\|_{L^2(M)}\|g\|_{L^2(M)}.
\end{align}
First we have from \eqref{Rif}, the $L^2$ boundedness of 
$\chi_i(\sqrt{-\Delta_i}-\lambda_i)$ ($i=1,2$), and Fubini's theorem  the following basic inequalities
$$\|R_{i,\lambda_i} f\|_{H^{k}(M_i, L^2(M_{3-i}))}=\|R_{i,\lambda_i} f\|_{L^{2}(M_{3-i}, H^k(M_{i}))}\lesssim_{N,k} \lambda_i^{k-N}\|f\|_{L^2(M)},\ k=0,\ldots,N,\ i=1,2,$$
$$\|R_{1,\lambda_1}R_{2,\lambda_2}f\|_{H^{k}(M)}
\lesssim \|R_{1,\lambda_1}R_{2,\lambda_2}f\|_{H^k(M_2, H^k(M_1))}\lesssim_{N,k} (\lambda_1\lambda_2)^{k-N}\|f\|_{L^2(M)},\ k=0,\ldots,N,$$
$$\|T_{i,\lambda_i}f\|_{L^2(M)}\lesssim \lambda_i^{-\frac{d_i-1}{2}}\|f\|_{L^2(M)},\ i=1,2,$$
$$\|T_{1,\lambda_1}T_{2,\lambda_2}f\|_{L^2(M)}\lesssim \lambda_1^{-\frac{d_1-1}{2}}\lambda_2^{-\frac{d_2-1}{2}}\|f\|_{L^2(M)},$$
$$\|T_{1,\lambda_1}R_{2,\lambda_2}f\|_{H^k(M_2, L^2(M_1))}\lesssim_{N,k} \lambda_1^{-\frac{d_1-1}{2}}\lambda_2^{-N}\|f\|_{L^2(M)},\ k=0,\ldots,N,$$
$$\|R_{1,\lambda_1}T_{2,\lambda_2}f\|_{L^2(M_2, H^k(M_1))}\lesssim_{N,k} \lambda_2^{-\frac{d_2-1}{2}}\lambda_1^{-N}\|f\|_{L^2(M)},\ k=0,\ldots,N.$$
\underline{Case I: \eqref{TTRR},\eqref{TRRT},\eqref{TRRR},\eqref{RRRR}.} These inequalities share the feature that no $T_{i,\lambda_i}$ and $T_{i,\mu_i}$ appear together for a same $i$, and they can be established by applying H\"older's inequality. For example, \eqref{TRRT} holds by writing 
$$\|T_{1,\lambda_1}R_{2,\lambda_2}f\ R_{1,\mu_1}T_{2,\mu_2}g\|_{L^2(M)}\leq\|T_{1,\lambda_1}R_{2,\lambda_2}f\|_{L^\infty(M_2, L^2(M_1))}\|R_{1,\mu_1}T_{2,\mu_2}g\|_{L^2(M_2,L^\infty(M_1))}$$
and applying Sobolev embedding. \\
\underline{Case II: \eqref{TTTR},\eqref{TRTR}.} These share the feature that some but not all pairs of $T_{i,\lambda_i}$ and $T_{i,\mu_i}$ appear together for a same $i$. We use (2.14) of \cite{BGT05m} as follows 
$$\|T_{1,\lambda_1}f_1\ T_{1,\mu_1}g_1\|_{L^2(M_1)}\lesssim \min(\lambda_1,\mu_1)^{\frac{d_1}{2}-1}(\lambda_1\mu_1)^{-\frac{d_1-1}{2}}\|f_1\|_{L^2(M_1)}\|g_1\|_{L^2(M_1)}.$$
For example, applying the above inequality, \eqref{TTTR} follows by writing 
$$\|T_{1,\lambda_1}T_{2,\lambda_2}f\ T_{1,\mu_1}R_{2,\mu_2}g\|_{L^2(M)}=\|T_{1,\lambda_1}T_{2,\lambda_2}f\ T_{1,\mu_1}R_{2,\mu_2}g\|_{L^2(M_2, L^2(M_1))}$$
and then 
$$\left\|\|T_{2,\lambda_2}f\|_{L^2(M_1)} \|R_{2,\mu_2}g\|_{L^2(M_1)}\right\|_{L^2(M_2)}
\leq \|T_{2,\lambda_2}f\|_{L^2(M)}\|R_{2,\mu_2}g\|_{L^\infty(M_2,L^2(M_1))},$$
combined with Sobolev embedding. \\
\underline{Case III: \eqref{TTTT}.} For this, we further recall \cite{BGT05m} to write $y_i\in V_i$ in geodesic polar coordinates as 
$y_i=\exp_i(r_i\omega_i)$, $r_i>0$, $\omega_i\in\mathbb{S}^{d_i-1}$, where $\exp_i$ stands for the exponential mapping centered at the origin of $V_i$, 
and define $\varphi_{i,r_i}(x_i,\omega_i)=\varphi_i(x_i,\exp_i(r_i\omega_i))$. Write $dy_i=\kappa_i(r_i,\omega_i)\ dr_i\ d\omega_i$, and we define the operator acting on functions $f_i$ on $\mathbb{S}^{d_i-1}$
$$(T^{r_i}_{i,\lambda_i}f_i)(x_i)=\int_{\mathbb{S}^{d_i-1}}
e^{i\lambda_i\varphi_{r_i}(x_i,\omega_i)}a_{i,r_i}(x_i,\omega_i,\lambda_i)f_i(\omega_i)\ d\omega_i,$$
where $a_{i,r_i}(x_i,\omega_i,\lambda_i)=\kappa_i(r_i,\omega_i)a_i(x_i,\exp_i(r_i\omega_i),\lambda_i)$.
For $f\in L^2(M_1\times M_2)$, set 
$$f_{r_1,r_2}(\omega_1,\omega_2)=f(\exp_1(r_1\omega_1), \exp_2(r_2\omega_2)),$$
we then have 
$$T_{1,\lambda_1}T_{2,\lambda_2}f=\int_0^\infty\!\!\!\int_0^\infty T^{r_1}_{1,\lambda_1}T^{r_2}_{2,\lambda_2}f_{r_1,r_2}\ dr_1\ dr_2.$$
Similarly, with $g_{q_1,q_2}(\omega_1,\omega_2)=g(\exp_1(q_1\omega_1), \exp_2(q_2\omega_2))$, we have 
$$T_{1,\lambda_1}T_{2,\lambda_2}f\ T_{1,\mu_1}T_{2,\mu_2}g=\int_{\varepsilon/C}^{C\varepsilon}\!\int_{\varepsilon/C}^{C\varepsilon}\!\int_{\varepsilon/C}^{C\varepsilon}\!\int_{\varepsilon/C}^{C\varepsilon} T^{r_1}_{1,\lambda_1}T^{r_2}_{2,\lambda_2}f_{r_1,r_2}\ T^{q_1}_{1,\mu_1}T^{q_2}_{2,\mu_2}g_{q_1,q_2}\ dr_1\ dr_2\ dq_1\ dq_2.$$
By Minkowski inequality, to prove \eqref{TTTT}, it suffices to show 
$$\|T^{r_1}_{1,\lambda_1}T^{r_2}_{2,\lambda_2}f\ T^{q_1}_{1,\mu_1}T^{q_2}_{2,\mu_2}g\|_{L^2(M)}\lesssim \min(|\lambda|,|\mu|)^{\frac{d}{2}-r}(\lambda_1\mu_1)^{-\frac{d_1-1}{2}}(\lambda_2\mu_2)^{-\frac{d_2-1}{2}}\|f\|_{L^2(\mathbb{S}^{d_1-1}\times \mathbb{S}^{d_2-1})}\|g\|_{L^2(\mathbb{S}^{d_1-1}\times \mathbb{S}^{d_2-1})}.$$
The following lemma is from \cite[Lemma 2.9]{BGT05m}. 
\begin{lem}\label{lemma2.9}
Let $i=1,2$. Let $\underline{\omega}_i^{(1)},\underline{\omega}_i^{(2)}$ be two points on $\mathbb{S}^{d_i-1}$. Then there exists a splitting of the variable $x_i=(t_i,z_i)\in\R\times\R^{d_i-1}$ and neighborhoods $U_{i,1},U_{i,2}$ in $\mathbb{S}^{d_i-1}$ of $\underline{\omega}_i^{(1)},\underline{\omega}_i^{(2)}$ respectively such that $\varphi_{i,r_i}(t_i,z_i,w_i)$ satisfies 
$$\left|\det\left(\frac{\partial^2\varphi_{i,r_i}(t_i,z_i,\omega_i)}{\partial z_{i}\partial \omega_{i}}\right)\right|\geq c>0,$$
for $\omega_i\in U_{i,1}\bigcup U_{i,2}$. 
\end{lem}
Now we rework some of the lemmas in \cite{BGT05m} under the current product manifold setting. 
The following is a product version of 
\cite[Lemma 2.10]{BGT05m}. 
\begin{lem}\label{lemma2.10}
Let $(t_i,z_i)\in\R\times\R^{d_i-1}$ be any local system of coordinate near $(0,0)$, $i=1,2$. Then the operator 
$$g\in L^2(\mathbb{S}^{d_1-1}\times \mathbb{S}^{d_2-1})\mapsto (T^{r_1}_{1,\lambda_1}T^{r_2}_{2,\lambda_2}g)(t_1,z_1,t_2,z_2)\in L^2(\mathbb{R}_{t_1}\times\mathbb{R}_{t_2}, L^\infty(\mathbb{R}^{d_1-1}_{z_1}\times \mathbb{R}^{d_2-1}_{z_2}))$$
is continuous with norm bounded by $\lesssim (\lambda_1\lambda_2)^{-\frac{1}{2}}$. 
\end{lem}
To prove the above lemma, it suffices to observe that the kernel of the $TT^*$ operator 
$T^{r_1}_{1,\lambda_1}T^{r_2}_{2,\lambda_2}(T^{r_1}_{1,\lambda_1}T^{r_2}_{2,\lambda_2})^*
$
which equals $T^{r_1}_{1,\lambda_1}(T^{r_1}_{1,\lambda_1})^*\cdot T^{r_2}_{2,\lambda_2}(T^{r_2}_{2,\lambda_2})^*$ due to commutativity, is the product of the kernels of $T^{r_1}_{1,\lambda_1}(T^{r_1}_{1,\lambda_1})^*$ and  $T^{r_1}_{2,\lambda_2}(T^{r_1}_{2,\lambda_2})^*$. The rest of the proof follows as that of \cite[Lemma 2.10]{BGT05m}. 
The next lemma is again a product version of \cite[Lemma 2.14]{BGT05m}. 
\begin{lem}\label{lemma2.14}
Under the assumptions of Lemma \ref{lemma2.9}, the operator 
$$g\in L^2(\mathbb{S}^{d_1-1}\times \mathbb{S}^{d_2-1})\mapsto (T^{r_1}_{1,\lambda_1}T^{r_2}_{2,\lambda_2}g)(t_1,z_1,t_2,z_2)\in L^\infty(\mathbb{R}_{t_1}\times\mathbb{R}_{t_2}, L^2(\mathbb{R}^{d_1-1}_{z_1}\times \mathbb{R}^{d_2-1}_{z_2}))$$
is continuous with norm bounded by $\lesssim \lambda_1^{-\frac{d_1-1}{2}}\lambda_2^{-\frac{d_2-1}{2}}$.
\end{lem}
This follows from the product version of the nondegenerate phase estimate as follows. 
\begin{lem}
Let us consider for $i=1,2$, $\varphi_i\in C^\infty(\mathbb{R}^{n_i}_{z_i}\times\mathbb{R}^{n_i}_{\omega_i})$ and $a_i\in C_0^\infty(\mathbb{R}^{n_i}_{z_i}\times\mathbb{R}^{n_i}_{\omega_i})$ such that 
$$(z_i,\omega_i)\in\t{supp}(a_i)\ \Rightarrow\ \det\left(\frac{\partial^2\varphi_i}{\partial z_i\partial \omega_i}(z_i,\omega_i)\right)\neq 0.$$
Define the operator $S_{i,\lambda_i}$
$$S_{i,\lambda_i}f_i(z_i)=\int_{\R^{n_i}}e^{i\lambda_i\varphi_i(z_i,\omega_i)}a_i(z_i,\omega_i)f_i(\omega_i)\ d\omega_i.$$
Then 
$$\|S_{1,\lambda_1}S_{2,\lambda_2}f\|_{L^2(\mathbb{R}_{z_1}^{n_1}\times \mathbb{R}_{z_2}^{n_2})}\lesssim \lambda_1^{-\frac{n_1}{2}}\lambda_2^{-\frac{n_2}{2}}\|f\|_{L^2(\mathbb{R}_{\omega_1}^{n_1}\times \mathbb{R}_{\omega_2}^{n_2})}.$$
\end{lem}
With the proof of the original nondegenerate phase estimate as in Chapter IX Section 1.1 of \cite{Ste93} in mind, the proof of the above product version is again an observation that the kernel of the $TT^*$ operator 
$S_{1,\lambda_1}S_{2,\lambda_2}(S_{1,\lambda_1}S_{2,\lambda_2})^*$ is the product of the kernels of $S_{1,\lambda_1}S_{1,\lambda_1}^*$ and $S_{2,\lambda_2}S_{2,\lambda_2}^*$.
We may now finish the proof. Assume that 
$$|\lambda|=\sqrt{\lambda_1^2+\lambda_2^2}\leq \sqrt{\mu_1^2+\mu_2^2}=|\mu|.$$
We write 
$$T^{r_1}_{1,\lambda_1}T^{r_2}_{2,\lambda_2}f
=\int_{S^{d_1-1}}\!\int_{S^{d_2-1}}e^{i\lambda_1\varphi_{1,r_1}(x_1,\omega_1)+i\lambda_2\varphi_{2,r_2}(x_2,\omega_2)}
 a_{1,r_1}(x_1,\omega_1,\lambda_1)a_{2,r_2}(x_2,\omega_2,\lambda_2)f(\omega_1,\omega_2)\ d\omega_1\ d\omega_2,$$
$$T^{q_1}_{1,\mu_1}T^{q_2}_{2,\mu_2}g
=\int_{S^{d_1-1}}\!\int_{S^{d_2-1}}e^{i\mu_1\varphi_{1,q_1}(x_1,\omega'_1)+i\mu_2\varphi_{2,q_2}(x_2,\omega'_2)}
 a_{1,q_1}(x_1,\omega'_1,\mu_1)a_{2,q_2}(x_2,\omega'_2,\mu_2)g(\omega'_1,\omega'_2)\ d\omega'_1\ d\omega'_2.$$
After a partition of unity, we can suppose that on the support of 
$a_{i,r_i}(x_i,\omega_i,\lambda_i)a_{i,q_i}(x_i,\omega'_i,\mu_i)$, $(\omega_i,\omega'_i)$ is close to a fixed point $(\underline{\omega}_i^{(1)}, \underline{\omega}_i^{(2)})\in \mathbb{S}^{d_i-1}\times \mathbb{S}^{d_i-1}$, $i=1,2$. We can therefore use the splitting $x_i=(t_i,z_i)$ of Lemma \ref{lemma2.9}, and estimate using H\"older's inequality, Lemma \ref{lemma2.10} and \ref{lemma2.14} that 
\begin{align*}
\|T^{r_1}_{1,\lambda_1}T^{r_2}_{2,\lambda_2}f\ T^{q_1}_{1,\mu_1}T^{q_2}_{2,\mu_2}g\|_{L^2_{t_1,t_2}L^2_{z_1,z_2}}
&\leq\|T^{r_1}_{1,\lambda_1}T^{r_2}_{2,\lambda_2}f\|_{L^2_{t_1,t_2}L^\infty_{z_1,z_2}}\|T^{q_1}_{1,\mu_1}T^{q_2}_{2,\mu_2}g\|_{L^\infty_{t_1,t_2}L^2_{z_1,z_2}}\\
&\lesssim (\lambda_1\lambda_2)^{-\frac{1}{2}}\mu_1^{-\frac{d_1-1}{2}}\mu_2^{-\frac{d_2-1}{2}}\|f\|_{L^2}\|g\|_{L^2}\\
&\lesssim |\lambda|^{\frac{d}{2}-2}(\lambda_1\mu_1)^{-\frac{d_1-1}{2}}(\lambda_2\mu_2)^{-\frac{d_2-1}{2}}\|f\|_{L^2}\|g\|_{L^2}.
\end{align*}
This finishes the proof of $r=2$ case. For higher $r$, we proceed similarly and would need derive bounds as in \eqref{TTTT} through \eqref{RRRR}. All of these bounds may be derived in a similar way as above, and it suffices to mention that in deriving bounds as in the Case II above, we need to use inductively the bounds as in Case III for a lower $r$. 
\end{proof}

We now prove some bilinear Strichartz estimates assuming Conjecture \ref{bilineareigen} to be true. We need the following counting estimates of number of representations of an integer by a positive definite integral quadratic form. 

\begin{lem}\label{counting}
Let $Q(\cdot)$ be a positive definite quadratic form of integral coefficients of $r$ variables. Then for $n\in\mathbb{Z}_{\geq 1}$, we have 
$$
 \#\{{\bf x}\in \mathbb{Z}^r:\ n=Q({\bf x})\}\left\{
\begin{array}{ll}
\lesssim_{\varepsilon>0} n^{r/2-1+\varepsilon}, & r=2,3,4;\\
\lesssim n^{r/2-1}, & r\geq 5.
\end{array}
\right.
$$
\end{lem}
\begin{proof}
For the special case when $Q$ is a sum of squares, the result is well-known; see \cite{Gro85}. For the general case, the first bound with the $\varepsilon$-loss may be proved by first reducing it to the case $r=2$, which is then reduced to counting algebraic integers with a fixed norm in an imaginary quadratic field, for which the standard divisor bound may be applied; see for example \cite[Lemma 8]{BP16}. For the second bound without the $\varepsilon$-loss, 
we may apply the circle method and write
$$
S_n:=\#\{{\bf x}\in \mathbb{Z}^r:\ n=Q({\bf x})\}
=\frac{1}{2\pi}\int_0^{2\pi}\underbrace{\sum_{{\bf x}\in\mathbb{Z}^r}\phi\left(Q({\bf x})/n\right)e^{it(Q({\bf x})-n)}}_{=:s_n(t)}\ dt
$$
using a bump function $\phi$ with $\phi(1)=1$. Let $N=\lfloor \sqrt{n}\rfloor$. 
The inner exponential sum $s_n(t)$ may be estimated by Weyl differencing to yield the bound on Farey arcs $\left\|\frac{t}{2\pi}-\frac{a}{q}\right\|\lesssim \frac{1}{qN}$
$$|s_n(t)|\lesssim\left(\frac{N}{\sqrt{q}(1+N\|\frac{t}{2\pi}-\frac{a}{q}\|^{1/2})}\right)^r.$$
Then we write using Farey dissection as in Section \ref{farey}
$$S_n=\sum_{Q,L}S_{Q,L}$$
where 
$$S_{Q,L}=\frac{1}{2\pi}\int_{\mathcal{M}_{Q,L}}s_n(t)\ dt.$$
Then the above bound tells 
$$|S_{Q,L}|\lesssim N^{\frac{r}{2}-1}L^{\frac{r}{2}-1}Q^{-\frac{r}{2}+2}.$$
Assuming $r\geq 5$, the above exponent of $Q$ is negative, thus summing over the dyadic integers $Q$ and $L$ yields the desired bound
$$S_n\leq\sum_{Q,L}|S_{Q,L}|\leq N^{r-2}.$$
\end{proof}

\begin{prop}\label{subcriticalbilinear}
Suppose $M$ is a symmetric space of compact type of dimension $d$ and rank $r\geq 2$. Suppose each irreducible component $M_0$ of $M$ has its dimension $d_0$ and rank $r_0$ such that $d_0\geq 3r_0$. 
Then Conjecture \ref{bilineareigen} (i) implies the following bilinear Strichartz estimate. For $N_1,N_2\geq 1$, we have 
\begin{align}\label{subcriticalbilinearStrichartz}
\|P_{N_1} e^{it\Delta} f_1 \cdot P_{N_2}e^{it\Delta}f_2\|_{L^2(\mathbb{T}\times M)}\lesssim_{\varepsilon} \min(N_1,N_2)^{\frac{d}{2}-1+\varepsilon}\|f_1\|_{L^2(M)}\|f_2\|_{L^2(M)}.
\end{align}
\end{prop}
\begin{proof}
Assume that $N_1\geq N_2$. 
We decompose the ambient Euclidean space $\mathfrak{a}^*$ where the weight lattice $\Lambda$ lies on, into  
$$\mathfrak{a}^*=\bigsqcup_j C_j,$$
where $C_j$'s are disjoint cubes of side-length $N_2$, and set 
$$P_{C_j}f=\sum_{\lambda\in \Lambda^+\cap C_j}P_{\lambda} f.$$
This decomposition is similar to that employed in Herr, Tataru and Tzvetlov's work \cite{HTT11} on tori.
Since the $P_{\lambda} f$'s are orthogonal to each other for distinct $\lambda$, it is clear that 
\begin{align*}
\|f\|_{L^2(M)}^2= \sum_{j}\|P_{C_j}f\|_{L^2(M)}^2.
\end{align*}
We claim the following almost orthogonality result
$$\|P_{N_1} e^{it\Delta} f_1\cdot  P_{N_2}e^{it\Delta}f_2\|^2_{L^2(\mathbb{T}\times M)}\lesssim \sum_{j}\|P_{C_j}P_{N_1} e^{it\Delta} f_1 \cdot P_{N_2}e^{it\Delta}f_2\|^2_{L^2(\mathbb{T}\times M)}.$$
This follows if $P_{C_i}P_{N_1}f_1 \cdot P_{N_2}f_2$ is orthogonal to $P_{C_j}P_{N_1}f_1 \cdot P_{N_2}f_2$ whenever $C_i$ and $C_j$ are at least a fixed number of $N_2$-cubes away from each other, or say the distance between $C_i$ and $C_j$ is bounded from below by a constant times $N_2$. 
By Lemma \ref{Fouriersupport}, $P_{C_k}P_{N_1}f_1\cdot P_{N_2}f_2$ is Fourier supported on $C_k+C\cdot [-N_2,N_2]^r$, $k=i,j$, so under the above conditions on $C_i$ and $C_j$ we have that $P_{C_k}P_{N_1}f_1 \cdot P_{N_2}f_2$, $k=i,j$ are of disjoint Fourier support and thus are orthogonal to each other. It now suffices to prove 
\begin{align*}
\|P_{C_j} e^{it\Delta} f_1 \cdot P_{N_2}e^{it\Delta}f_2\|_{L^2(\mathbb{T}\times M)}\lesssim_{\varepsilon} N_2^{\frac{d}{2}-1+\varepsilon}\|f_1\|_{L^2(M)}\|f_2\|_{L^2(M)}.
\end{align*}
We have 
\begin{align*}
\|P_{C_j} e^{it\Delta} f_1  \cdot P_{N_2}e^{it\Delta}f_2\|_{L^2(\mathbb{T}\times M)}^2
&=\left\|\sum_{\lambda_1\in \Lambda^+\cap C_j, \lambda_2\in \Lambda^+, |\lambda_2|_\rho\sim N_2}  e^{it(|\lambda_1|_\rho^2+|\lambda_2|_\rho^2)} P_{\lambda_1}f_1  \cdot  P_{\lambda_2}f_2\right\|_{L^2(\mathbb{T}\times M)}^2\\
&=\sum_{n\in\frac{2\pi}{\mathcal{T}}\mathbb{Z}}\left\|\sum_{\lambda_1\in \Lambda^+\cap C_j, \lambda_2\in \Lambda^+, |\lambda_2|_\rho\sim N_2, |\lambda_1|^2_\rho+|\lambda_2|_\rho^2=n} P_{\lambda_1}f_1  \cdot P_{\lambda_2}f_2\right\|_{L^2(M)}^2 \numberthis \label{proofnote2} \\ 
&\lesssim \sum_{n\in\frac{2\pi}{\mathcal{T}}\mathbb{Z}}
\#_{C_j,N_2,n}
\sum_{\lambda_1\in \Lambda^+\cap C_j, \lambda_2\in \Lambda^+, |\lambda_2|_\rho\sim N_2, |\lambda_1|^2_\rho+|\lambda_2|_\rho^2=n}\| P_{\lambda_1}f_1 \cdot P_{\lambda_2}f_2\|_{L^2(M)}^2, \numberthis \label{proofnote1}
\end{align*}
where $$\#_{C_j,N_2,n}=\#\{\lambda_1\in \Lambda^+\cap C_j,\ \lambda_2\in \Lambda^+:\ |\lambda_2|_\rho\sim N_2,\ |\lambda_1|^2_\rho+|\lambda_2|_\rho^2=n\}.$$
By Conjecture \ref{bilineareigen} (i), 
$$\| P_{\lambda_1}f_1 \cdot P_{\lambda_2}f_2\|_{L^2(M)}\lesssim_{\varepsilon} N_2^{\frac{d}{2}-r+\varepsilon}\|P_{\lambda_1}f\|_{L^2(M)}\|P_{\lambda_2}f_2\|_{L^2(M)}.$$
Thus it suffices to observe
\begin{align*}
\#_{C_j,N_2,n}&\lesssim \#\{\lambda_1\in \Lambda^+\cap C_j\}\cdot\max\#\{\lambda_2\in\Lambda^+:\ |\lambda_2|_\rho\sim N_2,\ |\lambda_2|_\rho^2=n-|\lambda_1|_\rho^2\}\\
&\lesssim_{\varepsilon} N_2^{r}\cdot N_2^{r-2+\varepsilon}
\end{align*}
where the second term is estimated by Lemma \ref{counting}, noting the rationality of the Killing form as described in Lemma \ref{rationality}.  
\end{proof}

We also have the following critical refinement of the above proposition. 

\begin{prop}\label{criticalbilinear}
Assume $r\geq 3$ and $d_0>3r_0$ for each irreducible component $M_0$ of $M$. Then Conjecture \ref{bilineareigen} (ii) implies the following bilinear Strichartz estimate for some $\delta>0$
\begin{align}\label{bilineardelta}
\|P_{N_1} e^{it\Delta} f_1   \cdot P_{N_2}e^{it\Delta}f_2\|_{L^2(M)}\lesssim \left(\frac{N_2}{N_1}+\frac{1}{N_2}\right)^{\delta}N_2^{\frac{d}{2}-1}\|f_1\|_{L^2(M)}\|f_2\|_{L^2(M)},
\end{align}
where $N_2\leq N_1$. 
\end{prop}
\begin{proof}
Using the same decomposition into cubes $\mathfrak{a}^*=\bigsqcup_j C_j$ as in the proof of Proposition \ref{subcriticalbilinear}, we first reduce to the following estimate 
\begin{align*}
\|P_{C_j}P_{N_1} e^{it\Delta} f_1  \cdot  P_{N_2}e^{it\Delta}f_2\|_{L^2(M)}\lesssim \left(\frac{N_2}{N_1}+\frac{1}{N_2}\right)^{\delta}N_2^{\frac{d}{2}-1}\|f_1\|_{L^2(M)}\|f_2\|_{L^2(M)}. 
\end{align*}
Let $\xi_{j}$ be the center of $C_j$. We make a further decomposition 
$$C_j=\bigsqcup_k R_{j,k}$$
of each cube $C_j$ into strips
$$R_{j,k}=\{\xi\in C_j: \ \xi\cdot\xi_{j}=|\xi_j|kM\}$$ 
of width $M$, where 
$$|k|\sim \frac{N_1}{M}.$$
Again, a similar decomposition was employed in \cite{HTT11}.  
A straightforward calculation shows that 
$P_{R_{j,k}}P_{N_1} e^{it\Delta} f_1$ are almost orthogonal to each other in the space $L^2(\mathbb{T})$ with respect to the time variable, thus it then suffices to prove 
\begin{align*}
\|P_{R_{j,k}}P_{N_1} e^{it\Delta} f_1  \cdot  P_{N_2}e^{it\Delta}f_2\|_{L^2(M)}\lesssim \left(\frac{N_2}{N_1}+\frac{1}{N_2}\right)^{\delta}N_2^{\frac{d}{2}-1}\|f_1\|_{L^2(M)}\|f_2\|_{L^2(M)}. 
\end{align*}
To proceed, we provide two different approaches.\\
\underline{Approach 1.}
Arguing similarly as in the proof of Proposition \ref{subcriticalbilinear}, we end up with the estimate 
\begin{align*}
\|P_{R_{j,k}}P_{N_1} e^{it\Delta} f_1 \cdot   P_{N_2}e^{it\Delta}f_2\|_{L^2(M)}^2\lesssim \sum_{n\in\frac{2\pi}{\mathcal{T}}\mathbb{Z}}
\#_{R_{j,k},N_2,n}
\sum_{\substack{\lambda_1\in \Lambda^+\cap R_{j,k}, \lambda_2\in \Lambda^+, \\ |\lambda_2|_\rho\sim N_2, |\lambda_1|^2_\rho+|\lambda_2|_\rho^2=n} }\| P_{\lambda_1}f_1 \cdot P_{\lambda_2}f_2\|_{L^2(M)}^2,
\end{align*}
where $$\#_{R_{j,k},N_2,n}=\#\{\lambda_1\in \Lambda^+\cap R_{j,k},\ \lambda_2\in \Lambda^+:\ |\lambda_2|_\rho\sim N_2,\ |\lambda_1|^2_\rho+|\lambda_2|_\rho^2=n\}.$$
Observe that if we assume $r\geq 5$, we may apply Lemma \ref{counting} to estimate
\begin{align*}
\#_{R_{j,k},N_2,n}&\lesssim \#\{\lambda_1\in \Lambda^+\cap R_{j,k}\}\cdot\max\#\{\lambda_2\in\Lambda^+:\ |\lambda_2|_\rho\sim N_2,\ |\lambda_2|_\rho^2=n-|\lambda_1|_\rho^2\}\\
&\lesssim MN_2^{r-1}\cdot N_2^{r-2}.
\end{align*}
This combined with Conjecture \ref{bilineareigen} (ii) proves \eqref{bilineardelta} for $\delta=1/2$. \\
\underline{Approach 2.} This approach replaces application of Lemma \ref{counting} by Strichartz estimates on tori; such a method has been employed in \cite{HS15}. Instead of applying Cauchy-Schwarz as in \eqref{proofnote1}, we use the triangle inequality to arrive at
\begin{align*}
\|P_{R_{j,k}}P_{N_1} e^{it\Delta} f_1 \cdot   P_{N_2}e^{it\Delta}f_2\|_{L^2(M)}^2\lesssim N_2^{d-2r}\sum_{n\in\frac{2\pi}{\mathcal{T}}\mathbb{Z}}
\left(\sum_{\substack{\lambda_1\in \Lambda^+\cap R_{j,k}, \lambda_2\in \Lambda^+, \\ |\lambda_2|_\rho\sim N_2, |\lambda_1|^2_\rho+|\lambda_2|_\rho^2=n} }\|P_{\lambda_1}f\|_{L^2(M)}\|P_{\lambda_2}f_2\|_{L^2(M)}\right)^2,
\end{align*}
after an application of Conjecture \ref{bilineareigen} (ii). The sum on the right can be rewritten as 
$$
\left\|\left(\sum_{\lambda_1\in\Lambda^+\cap R_{j,k}}e^{it|\lambda_1|_\rho^2}\|P_{\lambda_1}f\|_{L^2(M)}\right)\left(\sum_{|\lambda_2|_\rho\sim N_2}e^{it|\lambda_2|_\rho^2}\|P_{\lambda_2}f\|_{L^2(M)}\right)\right\|_{L^2(\mathbb{T})}
^2$$
which is then estimated by 
\begin{align}\label{doubleStri}
\lesssim \left(\underbrace{\left\|\sum_{\lambda_1\in\Lambda^+\cap R_{j,k}}e^{it|\lambda_1|_\rho^2}\|P_{\lambda_1}f\|_{L^2(M)}\right\|_{L^{p_1}(\mathbb{T})}}_{I_1(p_1)}\cdot \underbrace{\left\|\sum_{|\lambda_2|_\rho\sim N_2}e^{it|\lambda_2|_\rho^2}\|P_{\lambda_2}f\|_{L^2(M)}\right\|_{L^{p_2}(\mathbb{T})}}_{I_2(p_2)}\right)^2
\end{align}
for some $p_1,p_2$ such that $1/p_1+1/p_2=1/2$. We have the Strichartz estimate on rational tori
\begin{align}\label{StriTori}
\left\|\sum_{|\lambda|\sim N}e^{it|\lambda|^2+i(\lambda,x)}a_\lambda\right\|_{L^p(\mathbb{T}\times \mathbb{T}^r)}\lesssim N^{\frac{r}{2}-\frac{r+2}{p}}\|a_\lambda\|_{l^2}
\end{align}
for all $p>2(r+2)/r$, which was proved in \cite{BD15} (see also \cite{Zha20} for the $\varepsilon$-loss removal). We estimate  
\begin{align*}
\left\|\sum_{\lambda\in C_j}e^{it|\lambda|_\rho^2}a_\lambda\right\|_{L^p(\mathbb{T})}
\lesssim \left\|\sum_{\lambda\in C_j}e^{it|\lambda|^2+i(\lambda,x)}a_\lambda\right\|_{L^p(\mathbb{T}, L^\infty(\mathbb{T}^r))}.
\end{align*}
By Galilei invariance, we have 
$$\left\|\sum_{\lambda\in C_j}e^{it|\lambda|^2+i(\lambda,x)}a_\lambda\right\|_{L^p(\mathbb{T}, L^\infty(\mathbb{T}^r))}
=\left\|\sum_{|\lambda|\sim N}e^{it|\lambda|^2+i(\lambda,x)}\tilde{a}_{\lambda}\right\|_{L^p(\mathbb{T}, L^\infty(\mathbb{T}^r))}.$$
where $\tilde{a}_\lambda=a_{\lambda+\lambda_0}$ for some $\lambda_0\in \Lambda$ near the center of $C_j$. 
Apply Bernstein type inequalities on tori (valid on any compact manifold \cite[Corollary 2.2]{BGT04}) and the above Strichartz estimate \eqref{StriTori}, the above is bounded by 
$$\lesssim  N^{\frac{r}{p}}\left\|\sum_{|\lambda|\sim N}e^{it|\lambda|^2+i(\lambda,x)}\tilde{a}_\lambda\right\|_{L^p(\mathbb{T}\times \mathbb{T}^r)}\lesssim N^{\frac{r}{2}-\frac{2}{p}}\|a_\lambda\|_{l^2}.$$
If $r\geq 3$, the above estimate is valid for any $p>10/3$; thus we  bound each of the terms in \eqref{doubleStri} for $p_1,p_2>10/3$ by 
$$I_1(p_1)\lesssim N_2^{\frac{r}{2}-\frac{2}{p_1}}\|f_1\|_{L^2(M)}, \ I_2(p_2)\lesssim N_2^{\frac{r}{2}-\frac{2}{p_2}}\|f_2\|_{L^2(M)}.$$
For $I_1$, we further bound 
$$I_1(\infty)\lesssim \#\{\lambda_1\in\Lambda^+\cap R_{jk}\}^{\frac{1}{2}}\|f_1\|_{L^2(M)}
\lesssim M^{\frac{1}{2}}N_2^{\frac{r-1}{2}}\|f_1\|_{L^2(M)}.$$
By interpolation, for any $p_1>10/3$, there exists $\delta>0$ such that  
$$I_1(p_1)\lesssim \left(\frac{M}{N_2}\right)^\delta N_2^{\frac{r}{2}-\frac{2}{p_1}}\|f_1\|_{L^2(M)}.$$
Then \eqref{doubleStri} may be bounded by 
$\lesssim \left(\frac{M}{N_2}\right)^{2\delta} N_2^{2r-2}\|f_1\|_{L^2(M)}^2\|f_2\|_{L^2(M)}^2$, which concludes the proof. 

\end{proof}

We may slightly generalize Proposition \ref{subcriticalbilinear} and \ref{criticalbilinear} as follows to treat spaces which have toric components. 

\begin{prop}\label{mixed}
Suppose $M=\mathbb{T}^{r_0}\times M'$ where $\mathbb{T}^{r_0}$ is a rational torus of rank $r_0$ and $M'$ is a symmetric space of compact type. Suppose the rank $r$ of $M$ is at least 3. \\
(i) Suppose any irreducible component of $M'$ has the dimension $d_0$ and rank $r_0$ such that $d_0\geq 3r_0$. Then \eqref{subcriticalbilinearStrichartz} holds under Conjecture \ref{bilineareigen} (i). \\
(ii) Suppose any irreducible component of $M'$ has the dimension $d_0$ and rank $r_0$ such that $d_0> 3r_0$. Then \eqref{bilineardelta} holds under Conjecture \ref{bilineareigen} (ii). 
\end{prop}
\begin{proof}
We first prove (i). Let the dimension and rank of $M'$ be respectively $d'$ and $r'$. 
Suppose $N_1\geq N_2$. 
The spectral parameters for $M$ lie in the space $\Gamma=\Z^{r_0}\times \Lambda^+$ where $\Lambda^+$ is the set of dominant weights associated to $M'$. As in the proof of Proposition \ref{subcriticalbilinear}, we still decompose the ambient Euclidean space in which $\Gamma$ lies into disjoint cubes and proceed as there. The modification needed is on the estimation of $\|P_{C_j} e^{it\Delta} f_1  \cdot P_{N_2}e^{it\Delta}f_2\|_{L^2(\mathbb{T}\times M)}$. For $f\in L^2(\T^{r_0}\times M')$, its Fourier series reads
$$f(x,y)=\sum_{\xi\in\Z^{r_0},\lambda\in\Lambda^+}e^{i(\xi,x)}P_{\lambda}\hat{f}(\xi,y).$$
Then we proceed as in Approach 2 of the proof of Proposition \ref{criticalbilinear} as follows. We have 
\begin{align*}
\|P_{C_j} e^{it\Delta} f_1  \cdot P_{N_2}e^{it\Delta}f_2\|_{L^2(\mathbb{T}\times M)}^2
&=\left\|\sum_{\substack{ (\xi_1,\lambda_1)\in \Gamma\cap C_j 
\\ (\xi_2,\lambda_2)\in \Gamma \\ |\xi_2|^2+|\lambda_2|_\rho^2\sim N^2_2}}  e^{it(|\xi_1|^2+|\lambda_1|_\rho^2+|\xi_2|^2+|\lambda_2|_\rho^2)+i(x,\xi_1+\xi_2)} P_{\lambda_1}\widehat{f_1}  \cdot  P_{\lambda_2}\widehat{f_2}\right\|_{L^2(\mathbb{T}\times \T^{r_0}\times M')}^2\\
&=\sum_{n\in\frac{2\pi}{\mathcal{T}}\mathbb{Z},\ \xi\in\Z^{r_0}}
\left\|\sum_{|\xi_1|^2+|\lambda_1|_\rho^2+|\xi_2|^2+|\lambda_2|_\rho^2=n,\ \xi_1+\xi_2=\xi}P_{\lambda_1}\widehat{f_1}  \cdot  P_{\lambda_2}\widehat{f_2}
\right\|^2_{L^2(M')}\\
&\leq\sum_{n\in\frac{2\pi}{\mathcal{T}}\mathbb{Z},\ \xi\in\Z^{r_0}}
\left(\sum_{|\xi_1|^2+|\lambda_1|_\rho^2+|\xi_2|^2+|\lambda_2|_\rho^2=n,\ \xi_1+\xi_2=\xi}\|P_{\lambda_1}\widehat{f_1}  \cdot  P_{\lambda_2}\widehat{f_2}\|_{L^2(M')}
\right)^2.
\end{align*}
Assuming Conjecture \eqref{bilineareigen} (i), the above is bounded by 
$$\lesssim_\varepsilon N_2^{d'-2r'+\varepsilon}\sum_{n\in\frac{2\pi}{\mathcal{T}}\mathbb{Z},\ \xi\in\Z^{r_0}}
\left(\sum_{|\xi_1|^2+|\lambda_1|_\rho^2+|\xi_2|^2+|\lambda_2|_\rho^2=n,\ \xi_1+\xi_2=\xi}\|P_{\lambda_1}\widehat{f_1}\|_{L^2(M')}  \| P_{\lambda_2}\widehat{f_2}\|_{L^2(M')}
\right)^2.$$
Then the sum on the right of the above may be rewritten as  
$$\left\|\left(\sum_{(\xi_1,\lambda_1)\in\Gamma\cap C_j}e^{it(|\xi_1|^2+|\lambda_1|_\rho^2)+i(x,\xi_1)} \|P_{\lambda_1}\widehat{f_1}\|_{L^2(M')}\right)\left(
\sum_{|\xi_2|^2+|\lambda_2|^2_\rho\sim N_2^2}e^{it(|\xi_2|^2+|\lambda_2|_\rho^2)+i(x,\xi_2)} \|P_{\lambda_2}\widehat{f_2}\|_{L^2(M')}
\right)\right\|_{L^2(\T\times\T^{r_0})}^2.$$
Then we may proceed exactly as in the Approach 2 of the proof of Proposition \ref{criticalbilinear}, by adding another variable in the above exponential sums and reducing to Strichartz estimates on rational tori. The conclusion is that the above is bounded by 
$$\lesssim_\varepsilon N_2^{2r-r_0-2+\varepsilon}\left(\sum_{(\xi_1,\lambda_1)\in\Gamma\cap C_j}\|P_{\lambda_1}\widehat{f_1}\|_{L^2(M')}^2\right)\left(\sum_{|\xi_2|^2+|\lambda_2|^2_\rho\sim N_2^2} \|P_{\lambda_2}\widehat{f_2}\|_{L^2(M')}^2\right)
\lesssim N_2^{2r-r_0-2}\|f_1\|_{L^2}\|f_2\|_{L^2},$$
which implies the desired estimate. Part (ii) follows in a similar way, by adapting the Approach 2 in a similar manner as above. 
\end{proof}

Using Proposition \ref{jointspectral}, we have the following theorem as a consequence of Proposition \ref{subcriticalbilinear}, \ref{criticalbilinear}, and \ref{mixed}.

\begin{thm}\label{Main2}
Suppose $f_i\in L^2(M)$ is spectrally localized in the band $[N_i,2N_i]$, $i=1,2$. \\
(i) Suppose either $M=M_1\times\cdots\times M_r$ is a product of rank-one symmetric spaces of compact type such that $r\geq 2$, or $M=\T^{r_0}\times M_1\times\cdots\times M_{r-r_0}$ is a product of a rational $r_0$-dimensional torus $\T^{r_0}$ and rank-one spaces such that $r\geq 3$, and in both cases we assume that each $M_i$ has the dimension at least 3. Then  
$$\|e^{it\Delta}f_1\ e^{it\Delta}f_2\|_{L^2(I\times M)}\leq C_\varepsilon\min(N_1,N_2)^{\frac{d}{2}-1+\varepsilon} \|f_1\|_{L^2(M)}\|f_2\|_{L^2(M)}.$$
(ii) Suppose either $M=M_1\times\cdots\times M_r$ is a product of rank-one spaces, or $M=\T^{r_0}\times M_1\times\cdots\times M_{r-r_0}$ is a product of a rational torus $\T^{r_0}$ and rank-one spaces, and in both cases we assume $r\geq 3$ and that each component $M_i$ has the dimension at least 4. Suppose $N_1\geq N_2$. Then for some $\delta>0$ 
$$\|e^{it\Delta}f_1\ e^{it\Delta}f_2\|_{L^2(I\times M)}\leq C\left(\frac{N_2}{N_1}+\frac{1}{N_2}\right)^{\delta} N_2^{\frac{d}{2}-1} \|f_1\|_{L^2(M)}\|f_2\|_{L^2(M)}.$$
\end{thm}

\begin{rem}
In the above theorem, the condition $r\geq 3$ is used in the proof to make sure that a Strichartz estimate as in \eqref{StriTori} holds true for some $p<4$. If it were for the quintic nonlinearity and thus the trilinear Strichartz estimate, then the condition $r\geq 2$ is enough to guarantee \eqref{StriTori} to hold for some $p<6$, which would imply the validity of the trilinear estimate; see for example \cite{HS15} on the rank-two case $\mathbb{S}^1\times \mathbb{S}^2$. 
\end{rem}

\section{Function spaces and applications}\label{fun}
In this section, we provide consequences to local well-posedness results for nonlinear Schr\"odinger equations of the linear and bilinear Strichartz estimates obtained in previous sections. 

\subsection{Lebesgue spaces}
\begin{prop}
Let $M$ be a compact manifold such that there exists a number $p_0$ such that \eqref{Stri} holds for all $p>p_0$.  
Suppose in the Cauchy problem \eqref{Cauchyproblem} $F$ is a polynomial in $u$ and its complex conjugate $\bar{u}$ of degree $\beta$ that satisfies $F(0)=0$. Then this problem is uniformly locally well-posedness in $H^s(M)$ for any $s>s_*=\frac{d}{2}-\frac{2}{\max(\beta-1,p_0)}$ in the followng sense. For any $s_0>s_*$, $s\geq s_0$, and any bounded subset $B$ of $H^s(M)$, there exists $p=p(s_0)>\max(\beta-1,p_0)$, $T=T(B)>0$ and a unique solution in $C([-T,T], H^s(M))\cap L^p([-T,T], L^\infty(M))$ such that the solution map 
$$B\ni u_0\mapsto u\in C([-T,T], H^s(M))\cap L^p([-T,T], L^\infty(M))$$
is Lipschitz continuous. 
\end{prop} 

\begin{proof}
This is a quick adaptation of the proof of Proposition 3.1 in \cite{BGT04}. Select $p>\max(\beta-1,p_0)$ such that $s>\frac{d}{2}-\frac{2}{p}$, and consider 
$$Y_T=C([-T,T],H^s(M))\cap L^p([-T,T],W^{\sigma,p}(M))$$
where $\sigma=s-\left(\frac{d}{2}-\frac{d+2}{p}\right)>\frac{d}{p}$. The result now follows as in the proof of \cite[Proposition 3.1]{BGT04}. 
\end{proof}

With Theorem \ref{Main} at hand, we now have the following local well-posedness result on arbitrary compact symmetric spaces. 

\begin{thm}
Suppose in the Cauchy problem \eqref{Cauchyproblem}, $F(u)$ is a polynomial function in $u$ and its complex conjugate $\bar{u}$ of degree $\beta$ such that $F(0)=0$. Then on any compact globally symmetric space $M$ of dimension $d$ and rank $r$,  
\eqref{Cauchyproblem} is uniformly well-posedness in $H^s(M)$ for any $s>\frac{d}{2}-\frac{2}{\max(\beta-1,2+8/r)}$. In particular, if $\beta\geq 3+8/r$, then \eqref{Cauchyproblem} is uniformly well-posedness in $H^s(M)$ for any $s>s_c=\frac{d}{2}-\frac{2}{\beta-1}$. 
\end{thm}

\subsection{$X^{s,b}$-spaces}
Let $M$ be an arbitrary compact Riemannian manifold. Define the following $X^{s,b}$-norm first explicitly introduced by Bourgain \cite{Bou93} on the setting of tori:
$$\|u\|_{X^{s,b}(\R\times M)}:=\left\|e^{-it\Delta}u(t,\cdot)\right\|_{H^b(\R,H^s(M))}.$$
For any $T>0$, define 
$$\|u\|_{X^{s,b}_T}:=\inf_{w\in X^{s,b}(\R\times M)}\left\{\|w\|_{X^{s,b}(\R\times M)}:\ w|_{[-T,T]}=u\right\}. $$
Then we have the following standard proposition, which we refer to Theorem 3 in \cite{BGT05} and its proof; see also \cite{Bou93}, \cite{GOW14}. 
\begin{prop}\label{subcriticalwellposedness}
Suppose \eqref{subcriticalbilinearStrichartz} holds. Suppose in \eqref{Cauchyproblem}, $F(u)$ equals any of $\pm|u|^2u$, $\pm u^3$, $\pm |u|^2\bar{u}$, $\pm \bar{u}^3$. Then \eqref{Cauchyproblem} is uniformly locally well-posed in $H^s(M)$ for any $s>s_c=\frac{d}{2}-1$ in the following sense. For any $s_0>s_c$, $s\geq s_0$, and any bounded subset $B$ of $H^s(M)$ there exist $b=b(s_0)>\frac{1}{2}$, $T=T(B)>0$ and a unique solution in $X^{s,b}_T(M)=C([-T,T], H^s(M))\cap X^{s,b}_T(M)$, such that 
the solution map 
$$B\ni u_0\mapsto u\in X^{s,b}_T(M)$$
is Lipschitz continuous. 
\end{prop}

\subsection{$U^p$-, $V^p$-spaces}
Let $1\leq p<\infty$. A step function $a:\R\to L^2(M)$ is called a $U^p$-atom, if 
$$a(t)=\sum_{k=1}^K \chi_{[t_{k-1},t_k)}a_k, \ \ \sum_{k=1}^K\|a_k\|^p_{L^2(M)}=1$$
for a partition $-\infty<t_0<\cdots<t_K\leq\infty$. The  $U^p(\mathbb{R}, L^2(M))$-norm is defined as the corresponding atomic space, i.e., 
$$\|u\|_{U^p(\mathbb{R}, L^2(M))}:=\inf\left\{\|\lambda_j\|_{l^1}:\ u=\sum_{j=1}^\infty \lambda_ja_j\t{ for }U^p\t{-atoms }a_j\right\}.$$
Then we define the $U^{2,s}(\R\times M)$-norm by 
$$\|u\|_{U^{2,s}(\R\times M)}:=\left(\sum_{N\geq 1}N^{2s}\|P_N(u(t))\|^2_{U^2(\mathbb{R}_t, L^2(M))}\right)^{\frac{1}{2}},$$
and the $X^s(\R\times M)$-norm by 
$$\|u\|_{X^s(\R\times M)}:=\|e^{-it\Delta}u\|_{U^{2,s}(\R\times M)},$$
and finally, for $T>0$, the $X^s_T(M)$-norm by 
$$\|u\|_{X^{s}_T(M)}:=\inf_{w\in X^{s}(\R\times M)}\left\{\|w\|_{X^{s}(\R\times M)}:\ w|_{[-T,T]}=u\right\}. $$

The following proposition is also standard, which may be proved by a straightforward adaptation of the argument in \cite{HS15} and \cite{HTT11} as well as in \cite{Her13}. Note that in the original treatment in \cite{HTT11} on the torus case, the above $U^{2,s}(\R\times M)$-norm is defined in a different way by taking the spectral localization to each individual spectral parameter instead of using the Littlewood-Paley projectors in order to use conveniently some orthogonality argument that is valid on tori. Then Herr in \cite{Her13} showed that this is not necessary and an alternate and general approach exists, which was then also used in \cite{HS15}.  

\begin{prop}\label{criticalwellposedness}
Suppose \eqref{bilineardelta} holds. Suppose in \eqref{Cauchyproblem}, $F(u)$ equals any of $\pm|u|^2u$, $\pm u^3$, $\pm |u|^2\bar{u}$, $\pm \bar{u}^3$. Then the Cauchy problem \eqref{Cauchyproblem} is locally well-posedness in $H^{s_c}(M)$ in the following sense. Let $s\geq s_c$. For any $\phi_*\in H^{s_c}(M)$, define 
$$B_\varepsilon(\phi_*):=\left\{\phi\in H^{s_c}(M):\ \|\phi-\phi_*\|_{H^{s_c}}<\varepsilon\right\}.$$
Then there exists $\varepsilon>0$ and $T=T(\phi_*)>0$ such that for any initial datum $\phi\in B_\varepsilon(\phi_*)\cap H^{s}(M)$, the Cauchy problem has a unique solution in $C([-T,T], H^{s}(M))\cap X^s_T(M)$, and the solution map 
$$B_\varepsilon(\phi_*)\cap H^s(M)\ni \phi\mapsto u\in C([-T,T], H^{s}(M))\cap X^s_T(M)$$
is Lipschitz continuous. 
\end{prop}

With Theorem \ref{Main2} at hand, we now have the following theorem as a consequence of Proposition \ref{subcriticalwellposedness} and \ref{criticalwellposedness}.

\begin{thm}
Suppose in \eqref{Cauchyproblem}, $F(u)$ equals any of $\pm|u|^2u$, $\pm u^3$, $\pm |u|^2\bar{u}$, $\pm \bar{u}^3$. Then on such $M$ as assumed in (i) of Theorem \ref{Main2}, we have that \eqref{Cauchyproblem} is locally well-posed in $H^s(M)$ for any $s>s_c=\frac{d}{2}-1$, and on such $M$ as assumed in (ii) of Theorem \ref{Main2}, we have that \eqref{Cauchyproblem} is  locally well-posed in $H^{s_c}(M)$. 
\end{thm}

\end{document}